\documentclass[10pt,reqno]{amsart}
\usepackage{amsfonts,amssymb,amsmath}
\usepackage{authblk}

\parskip        \baselineskip
\usepackage{aliascnt}
\usepackage[utf8]{inputenc}
\usepackage[english]{babel}
\usepackage[foot]{amsaddr}

\newtheorem{theorem}{Theorem}[section]
\newtheorem{lemma}[theorem]{Lemma}

\newcommand{\suchthat}{\;\ifnum\currentgrouptype=16 \middle\fi|\;}
\def\R{{\mathbb{R} }}
\renewcommand{\leq}{\leqslant}
\renewcommand{\geq}{\geqslant}

\DeclareMathOperator{\dvr}{div}
\newcommand{\tder}{\partial_t}
\newcommand{\LPND}[1]{L^{#1}_{\dvr}(\Omega)}
\newcommand{\LND}{\LPND{2}}
\newcommand{\WND}{W^{1,2}_{0,\dvr}(\Omega)}

\numberwithin{equation}{section}
\newtheorem{thm}{Theorem}
\numberwithin{thm}{section}

\newaliascnt{lemma}{thm}
\newtheorem{lem}[lemma]{Lemma}
\aliascntresetthe{lemma}

\newaliascnt{proposition}{thm}

\aliascntresetthe{proposition}

\newaliascnt{corollary}{thm}

\aliascntresetthe{corollary}

\newaliascnt{definition}{thm}
\newtheorem{mydef}[definition]{Definition}
\aliascntresetthe{definition}

\newaliascnt{remark}{thm}

\aliascntresetthe{remark}

\usepackage[colorlinks=true]{hyperref}
\def\eR{{\mathbb{R} }}
\def\eN{\mathbb{N}}

\newcommand\dt{\; \mathrm{d}t}

\title[Existence of weak solutions to a diffuse interface model for magnetic fluids]{Global existence of weak solutions to a diffuse interface model for magnetic fluids}
\author{Martin Kalousek$^{\dagger,\ddagger}$, Sourav Mitra$^\ddagger$, Anja Schl\"omerkemper$^\ddagger$}
\address{$^\dagger$ Institute of Mathematics, Czech Academy of Sciences, \v{Z}itn\'a 25, 11567 Prague, Czech Republic}

\address{$^\ddagger$ Institute of Mathematics, University of W{\"u}rzburg, Emil-Fischer-Str.\ 40, 97074 W{\"u}rzburg, Germany}

\email{martin.kalousek@mathematik.uni-wuerzburg.de}
\email{sourav.mitra@mathematik.uni-wuerzburg.de}
\email{anja.schloemerkemper@mathematik.uni-wuerzburg.de}

\begin{document}
		\maketitle
		\vspace{-1cm}
		\begin{abstract}
		    This article is devoted to the derivation and analysis of a system of partial differential equations modeling a diffuse interface flow of two Newtonian incompressible magnetic fluids. The system consists of the incompressible Navier-Stokes equations coupled with an evolutionary equation for the magnetization vector and the Cahn-Hilliard equations. We show global in time existence of weak solutions to the system using the time discretization method. 
		\end{abstract}
		\noindent{\bf Key words}. Cahn-Hilliard equations, diffuse interface model, global existence,  implicit time discretization, magnetization, incompressible Navier-Stokes equations, weak solution.
		\smallskip\\
		\noindent{\bf AMS subject classifications}. Primary: 35Q35 Secondary: 35D30, 76D05, 76T99.
		
		\section{Introduction}
		
		In this article we consider the flow of two viscous incompressible fluids with magnetic properties undergoing partial mixing in a bounded domain $\Omega\subset \R^d$, $d=2,3$ with its boundary having $C^{1,1}$ regularity. Let $T>0$ and $Q_T = \Omega \times (0,T)$. The fluid boundary is denoted by $\partial\Omega$ and $\Sigma_{T}$ denotes $\partial\Omega\times(0,T)$.  The main result of this article is Theorem~\ref{Thm:Main} on global existence of weak solutions to the following diffuse interface model coupling the incompressible Navier-Stokes equations, a Cahn-Hilliard dynamics and a gradient flow for the magnetization vector, which we derive in Section~\ref{sec:derivation}. In diffuse interface models a sharp interface is replaced by a thin interfacial layer where a partial mixing of two fluids is possible. This mixing is described by an order parameter $\phi:Q_T \to \R$, which in our case is the concentration difference of the two fluids. 
		Let $v:Q_T\to \R^d$ denote the mean fluid velocity, $p: Q_T\to \R$ the pressure, $M:Q_{T}\rightarrow\mathbb{R}^{3}$ the magnetization, and $\mu:Q_T \to \R$ the chemical potential. Then the system reads
	\begin{equation}\label{diffviscoelastic*}
	\begin{alignedat}{2}
	\partial_{t}v+(v\cdot\nabla)v-\nu\Delta v+\nabla p=&\mu\nabla\phi+\frac{\xi(\phi)}{\alpha^{2}}((|M|^{2}-1)M)\nabla M-\dvr(\xi(\phi)\nabla M)\nabla M&&\mbox{ in } Q_{T},\\
	\dvr v=&0&&\mbox{ in } Q_{T},\\
	\partial_{t}M+(v\cdot\nabla)M=&\dvr(\xi(\phi)\nabla M)-\frac{\xi(\phi)}{\alpha^{2}}(|M|^{2}-1)M&&\mbox{ in } Q_{T},\\
	\partial_{t}\phi+(v\cdot\nabla)\phi=&\Delta \mu&&\mbox{ in } Q_{T},\\
	\mu=&-\eta\Delta\phi+\frac{1}{\eta}(|\phi|^{2}-1)\phi+\xi'(\phi)\frac{|\nabla M|^{2}}{2}+\frac{\xi'(\phi)}{4\alpha^{2}}(|M|^{2}-1)^{2}&&\mbox{ in }Q_{T},\\
	v=0,\ \partial_{n}M=&0,\ \partial_{n}\phi=\partial_{n}\mu=0&&\mbox{ on } \Sigma_{T},\\
	 (v,M,\phi)(\cdot,0)=&(v_{0},M_{0},\phi_{0})&& \mbox{ in }\Omega,
	\end{alignedat}
	\end{equation} 
In system \eqref{diffviscoelastic*}, $\nu,$ $\alpha$ and $\eta$ are positive constants, where $\nu$ is the viscosity coefficient, $\alpha$ is a factor needed for a penalization of the saturation condition of the magnetization vector punishing the deviation of $|M|$ from 1, and $\eta$ corresponds to the thickness of the interfacial region. The function $\xi(\phi)$ is the mobility of the magnetization; we assume it to be non degenerate, i.e., having a positive lower bound, and both $\xi$ and $\xi'$ are bounded from above, cf.\ \eqref{XiAssum}.

The main result of this article (Theorem~\ref{Thm:Main}) is an existence result of global weak solutions to system~\eqref{diffviscoelastic*}. We present this theorem in Section~\ref{sec:mainresult}, where we also fix the functional framework and present the main ideas of the proof. After the derivation of the system in Section~\ref{sec:derivation}, we introduce a time-discretized system and show existence of solutions to this system in Section~\ref{sec2}. In Section~\ref{sec3} we pass from the time discretized model to the original system~\eqref{diffviscoelastic*} and prove the central existence result Theorem~\ref{Thm:Main}. In Section~\ref{sec:comments} we comment on potential extensions of our work to a setting where viscosity and mobility coefficients may depend on the order parameter $\phi$, and on difficulties that arise when coupling our system with an evolution equation for the deformation tensor. In the appendix 
we provide several supporting lemmas. 

In the remainder of this introduction we embed our work in the existing literature. Diffuse interface models without magnetization, involving fluids with matched densities date back to \cite{hohenberg}. The article \cite{gurtin} gives a continuum mechanical derivation of such a model based on the concept of microforces. For a review of this topic we also refer to \cite{anderson}. One of the first mathematical results for such a system can be found in \cite{star} which deals with the qualitative behavior and stability of stationary solutions of the system as $t\rightarrow\infty.$ \\
Global in time existence of weak solutions of such a model in both dimension 2 and 3 is proved in \cite{boyer}. The article \cite{boyer} also deals with the existence of strong solutions for non degenerate mobility and proves that the model under consideration admits a strong solution globally in time in dimension 2 and locally dimension 3. Later it is proved in \cite{abelsarma} that any weak solution to such a system becomes regular for large times and converges as $t\rightarrow\infty$ to a stationary solution to the system. The proof of the result in \cite{abelsarma} is based on a new regularity theory of Cahn-Hilliard equation in spaces of fractional time regularity and maximal regularity theory of Stokes system. Unlike our case both the articles \cite{boyer} and \cite{abelsarma} deal with a singular potential in the energy $\mathcal{E}_{mix}$ whereas we use a double well potential, see \eqref{freeenergy-mix} for details. In the articles \cite{boyer} and \cite{abelsarma} the use of a singular potential plays a crucial role in proving that the order parameter stays in a physically reasonable interval $[-1,1]$, which does not hold true in our setting. For some recent results on the diffuse interface models we would further like to quote the articles \cite{gal,Giorgini} and the references therein.\\
	Several diffuse interface models (without magnetization) including fluids of unmatched densities have been developed in the literature. For instance one can consult the articles \cite{lowengrub,boyer2,ding,AbelsGrun}. Mathematical analysis of the thermodynamically consistent model introduced in \cite{AbelsGrun} can be found in \cite{AbDeGa113} (the case of non degenerate mobility) and \cite{AbDeGa213} (the case of  degenerate mobility ). The readers can also consult \cite{AbelsFeireisl} for the existence of weak solution to a compressible diffuse interface model.\\
	The mathematical study of diffuse interface model with fluids having different magnetic behavior is quite new in the literature. The article \cite{Nochetto} derives a simplified phase field model for ferromagnetic fluids which involves incompressible Navier-Stokes equations, advection reaction equation for the magnetization and the  Cahn-Hilliard equation for the phase field. The article \cite{Nochetto} also provides an energy stable numerical scheme for the model they derive. The article \cite{yangmao} proposes and proves the existence of weak solution for a diffuse interface magnetohydrodynamic model which involves the incompressible Navier-Stokes equations, the Maxwell
	 equations of electromagnetism and the Cahn-Hilliard equations.\\
	 The model in \cite{yangmao} is different from ours since we consider a gradient flow equation for the magnetization and not the Maxwell's equations for the magnetic field. Further unlike \cite{yangmao}, in our case the magnetization $M$ enters into the Cahn-Hilliard dynamics. This leads to one of the main mathematical difficulties in the present article, since the presence of $|\nabla M|^{2}$ in the Cahn-Hilliard part, cf.\ \eqref{diffviscoelastic*}$_{5}$, refrains us from obtaining $L^{2}_{loc}(W^{2,2})$ (unlike  \cite{AbDeGa113} and \cite{AbDeGa213}) regularity for the order parameter $\phi.$ In fact we only obtain that $\phi\in C_{w}(W^{1,2}).$ This hinders the possibility to bootstrap the regularity of $\nabla M$ from the fact that $\dvr(\xi(\phi)\nabla M)\in L^{2}(Q_{T})$ (this is just a consequence of the energy estimate).\\ 
	 To the best of our knowledge the present article is the first one proving the existence of weak solutions for a system of the form  \eqref{diffviscoelastic*}. The model \eqref{diffviscoelastic*} shares some similarities with the ones considered in \cite{shenyang} and \cite{yuefeng} (two phase model involving nematic liquid crystals and a incompressible viscous fluid) but we recall that these articles are written from a modeling and numerical point of view whereas our goal is to prove a mathematical theory of existence.


\section{Functional framework, main existence result and ideas of its proof} \label{sec:mainresult}	

Before we give the definition of weak solutions and state our main existence result, we introduce the notations which will be used further. Throughout the article we denote by $c$ a generic constant which might vary from line to line. By $\cdot$, a centered dot, we denote the scalar product of vectors and matrices. We use the standard notation for Lebesgue spaces and Sobolev spaces on a domain $\Omega\subset\eR^d$, $d=2,3$, i.e., we write $L^{p}(\Omega), W^{s,p}(\Omega)$ respectively, for $p\in[1,\infty]$, $s\in(0,\infty]$. The spaces of $k$--times differentiable functions on $\Omega$ is denoted by $C^k(\overline\Omega)$ and $C^{k,\lambda}(\overline\Omega)$ stands for the subspace of $C^k(\overline{\Omega})$ consisting of functions whose $k$--th derivative is H\"older continuous with an exponent $\lambda\in(0,1]$ in $\Omega$. The subscript $c$ in the expression of a function space signifies the compactness of the support of the functions involved. Since we do not distinguish explicitly in the notation between a Banach space $X$ of scalar functions and a space of a vector-valued functions with $m$ components each of which belongs to $X$, we use the notation $ \|\cdot \|_{L^p(\Omega)}$, $ \|\cdot \|_{W^{s,p}(\Omega)}$, etc. For Banach spaces $X,Y$ we denote by  $X\hookrightarrow Y$ ($X\stackrel{C}{\hookrightarrow} Y$) the continuous (compact) embedding of $X$ to $Y$. By $X'$ we mean the dual to a Banach space $X$ and for the corresponding duality pairing $\left\langle\cdot,\cdot\right\rangle$ is used. By $C_w([0,T];X)$ we mean a subspace of $L^\infty(0,T;X)$ consisting of such $f$ for which the mapping $t\mapsto\left\langle \phi, f(t)\right\rangle$ is continuous on $[0,T]$ for each $\phi\in X'$. Further, we set 
	\begin{align*}
	\LND&=\overline{\{v\in C^\infty_c(\Omega):\ \dvr v=0\text{ in }\Omega}^{\|\cdot\|_{L^2}},\\
	\WND&=\overline{\{v\in C^\infty_c(\Omega):\ \dvr v=0\text{ in }\Omega}^{\|\cdot\|_{W^{1,2}}},\\
	W^{2,2}_n(\Omega)&=\{u\in W^{2,2}(\Omega): \partial_n u=0\text{ on }\partial\Omega\},\\
	V(\Omega)&=\WND\cap W^{2,2}(\Omega).
	\end{align*}

	Before we state the precise definition of a weak solution we make assumptions on the function $\xi:\eR\to\eR$ that read: 
	\begin{equation}\label{XiAssum}
	\begin{split}
	\xi\in C^1(\eR),\\
	0<c_1\leq\xi\leq c_2\text{ on }\eR, \text{ for some }c_1,c_2>0,\\
	\xi'\leq c_3\text{ on }\eR, \text{ for some }c_3>0.\
	\end{split}
	\end{equation}
	The following function provides an example of such a non degenerate function $\xi$. Set 
	$$\xi(\phi)=(1-\mathcal{H}_{\eta}(\phi))\xi_{1}+\xi_{2}\mathcal{H}_{\eta}(\phi),$$
	where $\xi_{1},\xi_{2}>0$ are the exchange constants for the individual fluids and $\mathcal{H}_{\eta}(x)=\frac{1}{1+e^{-\frac{x}{\eta}}}$ is a regularization of the Heaviside step function. Since $\mathcal{H}_{\eta}'$ is bounded, $\xi(\cdot)$ of course satisfies the assumptions \eqref{XiAssum}. Such a regularization of the Heaviside function is also used in \cite{Nochetto} and \cite{yangmao}.
	
	 At this moment, we are in a position to define the precise notion of weak solution to \eqref{diffviscoelastic*} whose existence is investigated in this paper. 
	\begin{mydef}[Definition of weak solutions]\label{DefWS}
	For given $(v_0,M_0,\phi_0)\in L^2_{\dvr}(\Omega)\times W^{1,2}(\Omega)\times W^{1,2}(\Omega)$ we call the quadruple $(v,M,\phi,\mu)$ possessing the regularity
	\begin{equation}\label{functionalspaces}
	\begin{split}
	v&\in C_w([0,T];L^2_{\dvr}(\Omega))\cap L^2(0,T;W^{1,2}_{0,\dvr}(\Omega)),\\
	M&\in C_w([0,T];W^{1,2}(\Omega))\cap W^{1,2}(0,T;L^{\frac{3}{2}}(\Omega)),\\
	\phi&\in C_w([0,T];W^{1,2}(\Omega))\cap C^{0}([0,T];L^{2}(\Omega)),\\
	\mu&\in L^2(0,T;W^{1,2}(\Omega)),
	\end{split}
	\end{equation}
	a weak solution to \eqref{diffviscoelastic*} if it satisfies 
	\begin{equation}\label{WeakForm}
	\begin{split}
	\int_{\Omega}v(t)\cdot\psi_1(t)-\int_{\Omega}v_0\cdot\psi_{1}(0)=&\int_0^t\int_{\Omega}\biggl(v\cdot\tder\psi_{1} -(v\cdot\nabla)v\cdot\psi_{1}-\nu\nabla v\cdot\nabla\psi_{1} -\nabla\mu\phi\cdot\psi_{1}\\
	&+\left(\frac{\xi(\phi)}{\alpha^{2}}\bigl((|M|^{2}-1)M\bigr)\nabla M
	-\dvr\bigl(\xi(\phi)\nabla M\bigr)\nabla M\right)\cdot\psi_{1}\biggr),\\
	\int_\Omega M(t)\cdot\psi_2(t)-\int_\Omega M_0\cdot \psi_{2}(0)=&\int_0^t\int_\Omega\biggl(M\cdot\tder\psi_{2}-(v\cdot\nabla) M\cdot\psi_{2}-\xi(\phi)\nabla M\cdot\nabla\psi_{2}\\
	&-\frac{1}{\alpha^2}\bigl(\xi(\phi)(|M|^2-1)M\bigr)\cdot\psi_{2}\biggr),\\
	\int_\Omega \phi(t)\psi_3(t)-\int_\Omega\phi_0\psi_3(0)=&\int_0^t\int_\Omega\left(\phi\tder\psi_{3}-(v\cdot\nabla)\phi\psi_{3}-\nabla \mu\cdot\nabla\psi_{3}\right),\\
	\int_0^t\left(\int_\Omega\mu\psi_{3}-\eta\int_\Omega\nabla\phi\cdot\nabla\psi_{3}\right)=&\int_0^t \int_\Omega\biggl(\frac{1}{\eta}(|\phi|^{2}-1)\phi+\xi'(\phi)\frac{|\nabla M|^{2}}{2}+\frac{\xi'(\phi)}{4\alpha^{2}}(|M|^{2}-1)^{2}\biggr)\psi_{3}
	\end{split}
	\end{equation}
	for all $t\in(0,T)$, for all $\psi_{1}\in C^{1}_{c}([0,T);V(\Omega))$, $\psi_{2}\in C^{1}_{c}([0,T);W^{1,2}(\Omega))$ and all \\
	$\psi_{3}\in C^{1}_{c}\left([0,T);W^{1,2}(\Omega)\cap L^\infty(\Omega)\right)$.
	The initial data are attained in the form 
	\begin{equation}\label{InitDataAtt}
	\lim_{t\to 0_+}\left(\|v(t)-v_0\|_{L^2(\Omega)}+\|M(t)-M_0\|_{W^{1,2}(\Omega)}+\|\phi(t)-\phi_0\|_{W^{1,2}(\Omega)}\right)=0.
	\end{equation}
	\end{mydef} 

	Having introduced all the necessary ingredients we formulate the main result of this article that deals with the existence of a weak solution to \eqref{diffviscoelastic*}.
	\begin{thm}\label{Thm:Main}
		Let $T>0$, $\Omega\subset\eR^d$ be a bounded domain of class $C^{1,1},$ the assumptions in \eqref{XiAssum} hold and the initial data $(v_0,M_0,\phi_0)\in \LND\times W^{1,2}(\Omega)\times W^{1,2}(\Omega)$ be given. Then there exits a weak solution to \eqref{diffviscoelastic*} in the sense of Definition~\ref{DefWS}.
	\end{thm}
	
	\noindent\textbf{Ideas of proof.} 
	The proof of Theorem~\ref{Thm:Main} is given in Sections~\ref{sec2} and \ref{sec3}. It relies on a time discretization scheme in order to construct solutions to suitably chosen approximative problems.
	In view of other works in which the coupling of quantities involved in the system is similar to \eqref{diffviscoelastic*} the time discretization scheme can be easily adopted in order to tackle the problem. Considering a sequence $0=t_0<t_1<\ldots<t_k<t_{k+1}<\ldots$, $k\in\eN_0$ we find a solution $(v_{k+1},M_{k+1},\mu_{k+1},\phi_{k+1})$ to a stationary problem \eqref{timediscretesystem} at the point $t_{k+1}$ employing the components $(v_k,M_k,\phi_k)$ of the solution to the same problem at the time $t_k$. 
	
	A crucial step in the development of the proof was to introduce a discretized version of \eqref{diffviscoelastic*} which is unconditionally stable or, in other words, which yields a discrete energy estimate of the form \eqref{discreteestimate}. For example, we discretize the term $\frac{\xi(\phi)}{\alpha^{2}}(|M|^{2}-1)M$ appearing in \eqref{diffviscoelastic*}$_{3}$ as
	\begin{align} \label{discr-good} 
	\frac{\xi(\phi)}{\alpha^{2}}(|M|^{2}-1)M\approx \frac{{\xi(\phi_{k})}}{\alpha^{2}}(|M_{k+1}|^{2}M_{k+1}-M_{k}),
	\end{align}
	cf.\ \eqref{timediscretesystem}$_{3}$. Such a discretization is inspired by the convex splitting scheme used in \cite{yangmao} for scaler valued functions. We adapted such a discretization ($i.e.$ \eqref{discr-good}) in vector settings (since $M\in\mathbb{R}^{3}$). The inequality \eqref{discr-good} along with the inequality 
	\begin{align*}
	    \frac{1}{4}\bigl(|A|^{2}-1\bigr)^{2}-\frac{1}{4}\bigl(|B|^{2}-1\bigr)^{2}+\frac{1}{4}\bigl(|A|^{2}-|B|^{2}\bigr)^{2}+\frac{1}{2}|A\cdot(A-B)|^{2}
		+\frac{1}{2}|A-B|^{2}\leqslant (A-B)\cdot \bigl(|A|^{2}A-B\bigr) 
	\end{align*}
	for any $A,B\in \mathbb{R}^{3}$, see Lemma~\ref{algebriclem}, plays a crucial role in the derivation of the desired energy-like estimate \eqref{discreteestimate}.
	The latter inequality is one of the key observations of the present article despite its easy proof. It is inspired by the convex splitting scheme used in \cite{yangmao} for scalar-valued functions, in which case an equality of the form \eqref{AlgIneq2} is obtained.
	
	To solve the time discrete system \eqref{timediscretesystem}, we show that the existence of its solution is equivalent to the existence of a fixed point of a certain nonlinear operator and the existence of this fixed point is proven via the Leray-Schauder fixed point theorem.
	
	Using the information about the existence of solutions at nodal points $\{t_k\}$ we define piecewise constant interpolants in Section \ref{sec3} and show that these interpolants approximate $(v,M,\phi,\mu)$ which solves the weak formulations \eqref{WeakForm}. To this end, we first show an energy type estimate satisfied by the interpolants and recover weak type convergences. Next in order to pass to the limit in the nonlinear terms we need to recover strong convergences of the interpolants, which is achieved by applying an Aubin-Lions type lemma. The key to pass to the limit in the weak formulation, especially to the approximates corresponding to the terms $\mbox{div}(\xi(\phi)\nabla M)$ (cf.\ \eqref{diffviscoelastic*}$_{1}$) and $\xi'(\phi)\frac{|\nabla M|^{2}}{2}$ (cf.\ \eqref{diffviscoelastic*}$_{5}$) is to obtain the strong convergence of $\{\nabla M^{N}\}$, where $M^{N}$ is the interpolant approximating $M$. This is achieved by exploiting the monotone structure of the magnetization equation. Finally, we show that the initial data are attained in a strong sense, which then yields global existence of weak solutions to system~\eqref{diffviscoelastic*}.

	\section{Derivation of the model and related discussion} \label{sec:derivation}

In this section we derive a mathematical model for the flow of diffuse interface Newtonian incompressible magnetic fluids, which leads to system~\eqref{diffviscoelastic*} for smooth enough fields. We follow an energetic variational approach, which has been applied to various materials in the literature, cf., e.g., \cite{liucomplex} (on the modeling of elastic complex fluids),  \cite{liusun} (on nematic liquid crystal flows) and \cite{liubenesova} (on magnetoviscoelastic flows). We refer to \cite{fengliu, GigaKirshteinLiu2017} for related reviews and the references included therein.
The energetic variational approach is based on the so-called energy dissipation law, the least action principle, the maximum dissipation principle and Newton's force balance law. The systems of partial differential equations derived are phrased in Eulerian coordinates, which is particularly useful in interphase problems.

The energy dissipation law reads $\tfrac{d}{dt} \mathcal{E}_{tot} = -\mathcal{D}$, where $\mathcal{E}_{tot}$ denotes the total energy functional and $\mathcal{D}$ the dissipation functional. The total energy is given as a sum of the kinetic energy $\mathcal{K}= \int_\Omega \tfrac12 |u|^2$, the mixing energy $\mathcal{E}_{mix}$ and $\mathcal{E}_{mag}$, which models magnetic effects in the fluids under consideration. The mixing energy is defined by 
	\begin{equation}\label{freeenergy-mix}
	\begin{split}
	\mathcal{E}_{mix}(\phi)= \frac{\eta}{2}\int_{\Omega}|\nabla\phi|^{2}+\frac{1}{4\eta}\int_{\Omega}(|\phi|^{2}-1)^{2},
		\end{split}
	\end{equation}
where $\phi$ is as before the order parameter and $\eta >0$. The gradient part (or the regularization part) is the approximation of the interfacial energy and the term $\frac{1}{4\eta}(|\phi|^{2}-1)^{2}$ is the usual Ginzburg–Landau double-well potential penalizing the deviation of $|\phi|$ from 1. Physically it is reasonable to consider that $\phi$ takes values $-1$ and $+1$ for the unmixed fluids; it is further expected that $\phi\in[-1,+1]$ once the partial mixing occurs. Indeed, as $\phi$ solves a fourth order parabolic equation (as we will derive the Cahn-Hilliard dynamics), no comparison principle is available and furthermore we are not using a singular potential (one recalls that we are working with a Ginzburg-Landau potential of polynomial type) we will not be able to guarantee that $\phi\in[-1,+1].$ For a related discussion on the mixing energy $\mathcal{E}_{mix}$ and the sharp interface limit $\eta\longrightarrow 0$ we refer the readers to \cite{shenliu}.\\
The magnetic contribution $\mathcal{E}_{mag}$ to the energy that we consider here is motivated from micromagnetics, cf., e.g., the recent review \cite{DiFratta-etal2019} and references therein. However, for the time being, we only consider the so-called exchange energy contribution, which reflects the tendency of the magnetization to align in one direction. In our setting, also the order parameter enters into the magnetic energy contributions, which allows to study a bi-fluid model with the fluids having different magnetic behavior. In micromagnetics one takes the saturation condition into account, which means that the modulus of the magnetization is constant. As typical in the mathematical literature we set this (saturation) constant equal to one. Here we take the saturation condition into account by considering a penalization term which punishes the deviation of $|M|$ from 1, cf., e.g., \cite[Section 1.2]{Kurzke}, \cite{chipotshafrir} or \cite{anjazab}. Then $\mathcal{E}_{mag}$ reads 
\begin{equation*}
	\begin{split}
	&\mathcal{E}_{mag}(\phi,M)=\int_{\Omega}\xi(\phi)\frac{|\nabla M|^{2}}{2}+\frac{1}{4\alpha^{2}}\int_{\Omega}\xi(\phi)(|M|^{2}-1)^{2},
	\end{split}
	\end{equation*}
where $\alpha>0$ is used to control the strength of the penalization. Here we include a coupling with the order parameter $\phi$ through the factor $\xi(\phi)$ with $\xi$ as in \eqref{XiAssum}, see also the example given there. This function replaces the degenerate function $\frac{1+\phi}{2}$ that is considered in, e.g., \cite{yuefeng} in a related situation for liquid crystals. In \cite{shenyang} yet another degenerate function is introduced, viz $\left(\frac{1+\phi}{2}\right)^{2}$, which is used in place of the non degenerate choice $\xi(\phi)$. The analysis of the current article only allows us to deal with a non degenerate function $\xi(\cdot)$ as in \eqref{XiAssum}, cf., e.g., the proof of the strong convergence \eqref{MNStrongC}, where we use the non degeneracy of the function $\xi(\cdot).$ \\
The dissipation functional that we consider in this article represents the assumed viscous properties of the fluids and reads $\displaystyle\mathcal{D} = \int_\Omega \nu\left|\frac{\nabla v + (\nabla v)^\top}{2} \right|^2$ with the viscosity constant $\nu>0$.

The action functional is given by $\displaystyle\int_0^t  \mathcal{K} - \left(\mathcal{E}_{mag}+\mathcal{E}_{mix}\right)$. Its variations with respect to the flow map (also sometimes refered to as the displacement) yields the evolution equation for the linear momentum, where, based on Newton's force balance, we add the dissipative force to the right hand side of the momentum equation; the dissipative force is obtained from a variation of $\mathcal{D}$ with respect to the rate of the flow, i.e. with respect to the velocity. Due to the assumption of incompressibility and zero boundary conditions, the symmetric gradient of $v$ can be rewritten to obtain the term $\nu \Delta v$ in the momentum equation.\\
The variation of the kinetic energy with respect to the flow map is standard and provides the terms $\tder v+(v\cdot\nabla)v +\nabla \pi$, where $\pi$ is the fluid pressure. For the variation of $\mathcal{E}_{mag}+\mathcal{E}_{mix}$ with respect to the flow map, we follow, e.g., \cite[Eqn.~(13)--(17)]{yuefeng} and assume that the order parameter and the magnetization are convected with the material point, i.e., there is no diffusion in $\phi$ and $M$. Hence, when applying the chain rule, the calculation reduces to a derivation of the stress tensor 
	\begin{equation*}
	\begin{split}
	\mathbb{S}  = - \frac{\partial \mathcal{E}_{mag}}{\partial \nabla M} \odot \nabla M  -\frac{\partial \mathcal{E}_{mix}}{\partial \nabla \phi} \otimes \nabla \phi
	  = -\xi(\phi)(\nabla M\odot\nabla M)-\eta(\nabla\phi\otimes\nabla\phi), 
	\end{split}
	\end{equation*}
	where we used the notation $(\nabla M\odot\nabla M)_{ij}=\sum_{k=1}^{3}(\nabla_{i}M_{k})(\nabla_{j}M_{k})$ and $(\nabla\phi\otimes\nabla\phi)_{ij}=\nabla_{i}\phi\nabla_{j}\phi$.\\
Hence the variation of the action functional together with the dissipative force contribution leads to the balance equation of linear momentum, which is given by the incompressible Navier-Stokes equations with a modified expression of the stress tensor due to the combined effect of the mixing and magnetic energy contributions: 
	\begin{align*}
	\begin{array}{lll}
	&\tder v+(v\cdot\nabla)v 
	= \nu \Delta v - \nabla \pi + \dvr\mathbb{S}
	&\mbox{ in }Q_{T},\\
	&\dvr v=0 &\mbox{ in } Q_{T}.
	\end{array}
 \end{align*}


Next, variations of the energy functionals with respect to the other physical fields, viz the magnetization and the order parameter, are calculated and applied in a gradient flow approach. This yields the evolution equations for these fields in the Eulerian framework. We assume that the magnetization vector follows the transport $M(x(t)) = M(x(X,t),t) = M_0(X)$ for any $t \geq 0$, $x=x(X,t)$ and  $M_0$ being the initial condition. Under the assumption that $\partial_n M =0$ on $\Sigma_T$, the evolution equation for the magnetization vector reads
	\begin{equation*}
    \partial_{t}M+(v\cdot\nabla)M=-\frac{\delta\mathcal{E}_{tot}}{\delta M}=\dvr(\xi(\phi)\nabla M)-\frac{\xi(\phi)}{\alpha^{2}}(|M|^{2}-1)M\qquad\mbox{ in }
 Q_{T},
	\end{equation*}
	where $\displaystyle\frac{\delta\mathcal{E}_{tot}}{\delta M}$ denotes the variational derivative of the total energy functional $\mathcal{E}_{tot}$ with respect to $M$. 
	
	Concerning the derivation of the governing equation for $\phi,$ we assume a generalized Fick’s law, i.e., the mass flux be proportional to the
	gradient of the chemical potential (see, e.g., \cite{Cahn,shenliu} for details), and the physically reasonable boundary condition  $\partial_{n}\phi=\partial_{n}\mu=0$. We then obtain the following Cahn-Hilliard equations:
	\begin{equation*}
	\begin{array}{lll}
	& \partial_{t}\phi+(v\cdot\nabla)\phi=\Delta \mu&\qquad\mbox{in } Q_{T}\\
	\end{array}
	\end{equation*}
	with $$\mu=\frac{\delta\mathcal{E}_{tot}}{\delta \phi} 
=-\eta\Delta\phi+\frac{1}{\eta}(|\phi|^{2}-1)\phi+\xi'(\phi)\frac{|\nabla M|^{2}}{2}+\frac{\xi'(\phi)}{4\alpha^{2}}(|M|^{2}-1)^{2}\,\,\mbox{in}\,\, Q_{T}.$$
In summary, the physically reasonable boundary conditions solved by the unknowns are
$$v=0,\quad\partial_{n}M=0,\quad \partial_{n}\phi=\partial_{n}\mu=0\quad\mbox{on }\Sigma_{T}.$$
Summarizing the above equations and adding boundary and initial conditions, we thus obtain the system
\begin{equation}\label{diffviscoelastic}
	\begin{alignedat}{2}
	\tder v+(v\cdot\nabla)v+\nabla \pi -\nu\Delta v=&-\dvr(\xi(\phi)(\nabla M\odot\nabla M))-\eta\dvr(\nabla\phi\otimes\nabla\phi)&&\mbox{ in }Q_{T},\\
	\dvr v=&0&&\mbox{ in } Q_{T},\\
	\partial_{t}M+(v\cdot\nabla)M=&\dvr(\xi(\phi)\nabla M)-\frac{\xi(\phi)}{\alpha^{2}}(|M|^{2}-1)M&&\mbox{ in }Q_{T},\\
	\partial_{t}\phi+(v\cdot\nabla)\phi=&\Delta \mu&&\mbox{ in }Q_{T},\\
	\mu=&-\eta\Delta\phi+\frac{1}{\eta}(|\phi|^{2}-1)\phi+\xi'(\phi)\frac{|\nabla M|^{2}}{2}+\frac{\xi'(\phi)}{4\alpha^{2}}(|M|^{2}-1)^{2}&&\mbox{ in } Q_{T},\\
	v=0,\ \partial_{n}M=&0,\ \partial_{n}\phi=\partial_{n}\mu=0&&\mbox{ on }\Sigma_{T},\\
	(v,M,\phi)(\cdot,0)=&(v_{0},M_{0},\phi_{0})&& \mbox{ in } \Omega.
	\end{alignedat}
	\end{equation}

Next we will show that system~\eqref{diffviscoelastic*} is obtained from system~\eqref{diffviscoelastic} for smooth enough fields. This step is inspired by \cite{AbDeGa113} and \cite{shenyang}. The reason for rewriting the system is purely mathematical. From an analytical perspective, some nonlinearities in the momentum equation  \eqref{diffviscoelastic}$_{1}$ cause difficulties when passing to the limit in an existence proof. By rewriting the system as below, some problematic nonlinearities can be replaced by terms with better compactness properties, cf.\ \eqref{addedexpression} below.

To reformulate the momentum equation \eqref{diffviscoelastic}$_{1}$, we observe the following pointwise identities
	\begin{equation}\label{identities}
	\begin{split}
	&\eta\dvr(\nabla\phi\otimes\nabla\phi)=\eta\Delta\phi\nabla\phi+\eta\nabla\left(\frac{|\nabla\phi|^{2}}{2}\right),\\
	& \dvr(\xi(\phi)(\nabla M\odot\nabla M))=\dvr(\xi(\phi)\nabla M)\nabla M+\xi(\phi)\nabla\left(\frac{|\nabla M|^{2}}{2}\right).
	\end{split}
	\end{equation}
	Using \eqref{diffviscoelastic}$_{5}$ together with \eqref{identities}, we obtain for smooth enough fields
	\begin{equation}\label{addedexpression}
	\begin{split}
	 \dvr(\xi(\phi)(\nabla M\odot\nabla M))+\eta\dvr(\nabla\phi\otimes\nabla\phi)=& -\mu\nabla\phi+\frac{1}{4\eta}\nabla(|\phi|^{2}-1)^{2}+\eta\nabla\left(\frac{|\nabla\phi|^{2}}{2}\right)\\
	&+\nabla\left(\xi(\phi)\frac{|\nabla M|^{2}}{2}\right)+\nabla\left(\frac{\xi(\phi)}{4\alpha^{2}}(|M|^{2}-1)^{2}\right)\\
	&-\frac{\xi(\phi)}{\alpha^{2}}((|M|^{2}-1)M)\nabla M+\dvr(\xi(\phi)\nabla M)\nabla M.
	\end{split}
	\end{equation}
	The terms on the right hand side of \eqref{diffviscoelastic}$_{1}$ no longer appear in \eqref{diffviscoelastic*}. They are replaced by the first, sixth and seventh terms appearing in the right hand side of \eqref{addedexpression}. We note that the remaining terms of \eqref{addedexpression} can be included in the new pressure because of their gradient structure, which finishes the derivation of \eqref{diffviscoelastic*}.\\
    We remark that the new term $\mu\nabla\phi$ certainly has better compactness properties compared to $\dvr(\nabla\phi\otimes\nabla\phi)$. Later, this observation will allow to pass to the limit in the nonlinear terms while proving existence of weak solutions to \eqref{diffviscoelastic*}. 
	 In the next section we consider a time discrete variant of system \eqref{diffviscoelastic*}, which will help to conclude the existence proof.

\section{Existence of weak solutions to a time discrete model}\label{sec2}
This section is devoted to the existence proof for weak solutions to a time discrete problem corresponding to system~\eqref{diffviscoelastic*}. In that direction let $h>0$ be a positive constant and
let 
\begin{equation}\label{assumk}
v_{k}\in \LND,\,\,M_{k}\in W^{2,2}_{n}(\Omega),\,\,\phi_{k}\in W^{2,2}_{n}(\Omega)
\end{equation} 
be the information at $t_{k},$ $k\in\mathbb{N}_{0}.$
The quadruple $(v_{k+1},M_{k+1},\phi_{k+1},\mu_{k+1}),$ solution at $t_{k+1},$ is determined as a weak solution to the following system
\begin{equation} \label{timediscretesystem}
\begin{alignedat}{2}
 \frac{v_{k+1}-v_{k}}{h}+(v_{k+1}\cdot\nabla)v_{k+1}+\nabla p_{k+1}-\mu_{k+1}\nabla\phi_{k}=&
-{\xi(\phi_{k})}(|M_{k+1}|^{2}M_{k+1}-M_{k})\nabla M_{k+1}&&\\
&-\dvr(\xi(\phi_{k})\nabla M_{k+1})\nabla M_{k+1}+\nu\Delta v_{k+1}&&\text{ in }\Omega\\
 \dvr v_{k+1}=&0 &&\text{ in }\Omega\\
\frac{M_{k+1}-M_{k}}{h}+(v_{k+1}\cdot\nabla)M_{k+1}=&\dvr(\xi(\phi_{k})\nabla M_{k+1}) &&\\
&-\frac{{\xi(\phi_{k})}}{\alpha^{2}}(|M_{k+1}|^{2}M_{k+1}-M_{k})&&\text{ in }\Omega\\
\frac{\phi_{k+1}-\phi_{k}}{h}+(v_{k+1}\cdot\nabla)\phi_{k}=&\Delta\mu_{k+1}&&\text{ in }\Omega\\
\mu_{k+1}+\eta\Delta\phi_{k+1}-\frac{1}{\eta}(\phi^{3}_{k+1}-\phi_{k})=&H_{0}(\phi_{k+1},\phi_{k})\frac{|\nabla M_{k+1}|^{2}}{2}&&\\
&+\frac{1}{4\alpha^{2}}{H_{0}(\phi_{k+1},\phi_{k})}(|M_{k+1}|^{2}-1)^{2}&&\text{ in }\Omega\\
 v_{k+1}=\partial_{n}M_{k+1}=\partial_{n}\phi_{k+1}&=\partial_{n}\mu_{k+1}=0&&\text{ on }\partial\Omega
\end{alignedat}
\end{equation}
where $H_{0}:\mathbb{R}\times\mathbb{R}\to \mathbb{R}$ is defined as
\begin{equation}\label{H0}
\ H_{0}(a,b)=\begin{cases}
\frac{\xi(a)-\xi(b)}{a-b} &\mbox{if}\qquad a\neq b,\\[4.mm]
\xi'(b) & \mbox{if}\qquad a= b.
\end{cases}
\end{equation}

\enlargethispage{-\baselineskip}

We define the notion of weak solution to the time discretized problem \eqref{timediscretesystem} as follows
\begin{mydef}\label{defweaksolta}[Weak solution to the problem \eqref{timediscretesystem}]
	Let \eqref{assumk} hold. The quadruple 
	\begin{equation}\label{regularityk1}
	\begin{array}{l}
	(v_{k+1},M_{k+1},\phi_{k+1},\mu_{k+1})\in \WND\times W^{2,2}_{n}(\Omega)\times W^{2,2}_{n}(\Omega)\times W^{1,2}(\Omega),
	\end{array}
	\end{equation}
	is a weak solution to system~\eqref{timediscretesystem} if the following identities are true
	\begin{equation}\label{identity1}
	\begin{split}
	&\int_{\Omega} \frac{v_{k+1}-v_{k}}{h}\cdot\widetilde\psi_{1}+\int_{\Omega}(v_{k+1}\cdot\nabla)v_{k+1}\cdot\widetilde\psi_{1}+\int_{\Omega}\nabla\mu_{k+1}\phi_{k}\cdot\widetilde\psi_{1}\\
	&-\int_{\Omega}\left(\frac{{\xi(\phi_{k})}}{\alpha^{2}}(|M_{k+1}|^{2}M_{k+1}-M_{k})\nabla M_{k+1}\right)\cdot\widetilde\psi_{1}+\int_{\Omega}\left(\mathrm{div}(\xi(\phi_{k})\nabla M_{k+1})\nabla M_{k+1}\right)\cdot\widetilde\psi_{1}\\
	&=-\nu\int_{\Omega}\nabla v_{k+1}\cdot\nabla\widetilde\psi_{1}
	\end{split}
	\end{equation}
	for all $\widetilde\psi_{1}\in \WND,$\\
	\begin{equation}\label{identity2}
	\begin{split}
	&\int_{\Omega}\frac{M_{k+1}-M_{k}}{h}\cdot\widetilde\psi_{2}+\int_{\Omega}(v_{k+1}\cdot\nabla)M_{k+1}\cdot\widetilde\psi_{2}\\
	& =\int_{\Omega}\left(\dvr\left(\xi(\phi_{k})\nabla M_{k+1}\right)-\frac{\xi(\phi_{k})}{\alpha^{2}}(|M_{k+1}|^{2}M_{k+1}-M_{k})\right)\cdot\widetilde\psi_{2}
	\end{split}
	\end{equation}
	for all $\widetilde\psi_{2}\in L^2(\Omega),$\\
	\begin{equation}\label{identity3}
	\begin{split}
	&\int_{\Omega} \frac{\phi_{k+1}-\phi_{k}}{h}\cdot\widetilde\psi_{3}+\int_{\Omega}(v_{k+1}\cdot\nabla)\phi_{k}\cdot\widetilde\psi_{3}=-\int_{\Omega}\nabla\mu_{k+1}\cdot\nabla\widetilde\psi_{3}
	\end{split}
	\end{equation}
	and 
	\begin{equation}\label{identity4}
	\begin{split}
	&\int_{\Omega}\mu_{k+1}\widetilde\psi_{3}-\int_{\Omega} H_{0}(\phi_{k+1},\phi_{k})\frac{|\nabla M_{k+1}|^{2}}{2}\widetilde\psi_{3}-\int_{\Omega}\frac{H_{0}(\phi_{k+1},\phi_{k})}{4\alpha^{2}}(|M_{k+1}|^{2}-1)^{2}\widetilde\psi_{3}\\
	&=\eta\int_{\Omega}\nabla\phi_{k+1}\cdot\nabla\widetilde\psi_{3}+\frac{1}{\eta}\int_{\Omega}(\phi^{3}_{k+1}-\phi_{k})\widetilde\psi_{3}
	\end{split}
	\end{equation}
	for all $\widetilde\psi_{3}\in W^{1,2}(\Omega).$
\end{mydef}
Next we state a crucial estimate that is inspired by the convex splitting scheme used in \cite{yangmao}. In \cite{yangmao} the scalar version \eqref{AlgIneq2} is stated whereas we prove here the inequality \eqref{AlgIneq1} for vector-valued functions. This estimate proves to be efficient to deal with the term $(M_{k+1}-M_{k})\cdot\frac{\xi(\phi_{k})}{\alpha^{2}}(|M_{k+1}|^{2}M_{k+1}-M_{k}) $ when showing the inequality \eqref{discreteestimate}.
	\begin{lemma}\label{algebriclem}
		Let $A,B\in \mathbb{R}^{3}$ and $a,b\in \mathbb{R}.$ The following relations are true
		\begin{align}
		& \frac{1}{4}\bigl(|A|^{2}-1\bigr)^{2}-\frac{1}{4}\bigl(|B|^{2}-1\bigr)^{2}+\frac{1}{4}\bigl(|A|^{2}-|B|^{2}\bigr)^{2}+\frac{1}{2}|A\cdot(A-B)|^{2} +\frac{1}{2}|A-B|^{2}\nonumber\\
		&\leqslant (A-B)\cdot \bigl(|A|^{2}A-B\bigr)\label{AlgIneq1},\\
		& \frac{1}{4}(a^{2}-1)^{2}-\frac{1}{4}(b^{2}-1)^{2}+\frac{1}{4}(a^{2}-b^{2})^2+\frac{1}{2}(a^{2}-ab)^{2}+\frac{1}{2}(a-b)^{2}\nonumber\\
		&= (a-b)(a^{3}-b)\label{AlgIneq2}.
		\end{align}
	\end{lemma}
	\begin{proof}
		The proof involves straight forward algebraic manipulation. One expands the left hand side of \eqref{AlgIneq1} to furnish:
		\begin{equation*}
		\begin{split}
		& \frac{1}{4}\bigl(|A|^{2}-1\bigr)^{2}-\frac{1}{4}\bigl(|B|^{2}-1\bigr)^{2}+\frac{1}{4}\bigl(|A|^{2}-|B|^{2}\bigr)^{2}+\frac{1}{2}|A\cdot(A-B)|^{2}+\frac{1}{2}|A-B|^{2}\\
		&= (A-B)\cdot (|A|^{2}A-B)-\frac{1}{2}|A|^{2}|B|^{2}+\frac{1}{2}|A\cdot B|^{2}.
		\end{split}
		\end{equation*}
		Since
		$$|A\cdot B|\leqslant |A||B|$$
		by the Cauchy-Schwartz inequality, \eqref{AlgIneq1} follows.\\
		Identity \eqref{AlgIneq2} is taken from \cite{yangmao}.
		\end{proof}
\begin{thm}\label{existenceweaksoldiscrete}[Existence of weak solution to the problem \eqref{timediscretesystem}]
	Let \eqref{XiAssum} and \eqref{assumk} hold. Then there exists a quadruple $(v_{k+1},M_{k+1},\phi_{k+1},\mu_{k+1})$ which satisfies \eqref{regularityk1} and solves the integral identities \eqref{identity1}--\eqref{identity4}. Moreover, the following discrete version of the energy estimate holds
	\begin{equation}\label{discreteestimate}
	\begin{split}
	&  E_{tot}(v_{k+1},M_{k+1},\phi_{k+1})+h\nu\int_{\Omega}|\nabla v_{k+1}|^{2}
	+h\int_{\Omega}|\nabla\mu_{k+1}|^{2}\\ &+h\int_{\Omega}\left|\mathrm{div}(\xi(\phi_{k})\nabla M_{k+1})-\frac{\xi(\phi_{k})}{\alpha^{2}}(|M_{k+1}|^{2}M_{k+1}-M_{k})\right|^{2}
	 \leqslant E_{tot}(v_{k},M_{k},\phi_{k}),
	\end{split}
	\end{equation} 
	where
	\begin{equation}\label{defEtot}
	\begin{split}  
	E_{tot}(v,M,\phi)=\frac{1}{2}\int_{\Omega}|v|^{2}+\frac{1}{2}\int_{\Omega}\xi(\phi)|\nabla M|^{2}+\frac{1}{4\alpha^{2}}\int_{\Omega}\xi(\phi)(|M|^{2}-1)^{2}
	+\frac{\eta}{2}\int_{\Omega}|\nabla\phi|^{2}+\frac{1}{4\eta}\int_{\Omega}(\phi^{2}-1)^{2}.
	\end{split}
	\end{equation}
\end{thm}
\begin{proof}
	The proof is divided in two steps. Although it is natural to prove first the existence of a weak solution and then to obtain the energy dissipation inequality \eqref{discreteestimate} as a special case of an inequality derived in the existence proof, we choose to present them in a reverse way. The reason is that the approach we took consists in the application of Lemma~\ref{algebriclem} and identity  \eqref{algebricidnty} in both steps but its presentation requires less space and is easier to understand in the proof of the energy inequality.
	The first step, i.e., the obtainment of the estimate \eqref{discreteestimate} is contained in Section \ref{discreteenergyestimatesection} and the second step, i.e., the proof of the existence of weak solutions to \eqref{timediscretesystem} is included in Section \ref{existencetimediscrete}. We emphasize that the proof of the second step is independent of the first one.
	
	\subsection{Any weak solution \texorpdfstring{$(v_{k+1},M_{k+1},\phi_{k+1},\mu_{k+1})$}{quadruple} of \texorpdfstring{\eqref{timediscretesystem}}{sref} in the sense of Definition~\ref{defweaksolta} satisfies \texorpdfstring{\eqref{discreteestimate}}{soref}}\label{discreteenergyestimatesection}
	To this end we set $\widetilde\psi_1=v_{k+1}$ in \eqref{identity1},  $\widetilde\psi_2=-\mbox{div}(\xi(\phi_{k})\nabla M_{k+1})+\frac{\xi(\phi_{k})}{\alpha^{2}}(|M_{k+1}|^{2}M_{k+1}-M_{k})$ in \eqref{identity2}, $\widetilde\psi_3=\mu_{k+1}$ in  \eqref{identity3} and  $\widetilde\psi_3=-\frac{\phi_{k+1}-\phi_{k}}{h}$ in \eqref{identity4}, employ the definition \eqref{H0} of $H_{0}(\cdot,\cdot)$ and add up the resulting expressions to furnish:
	\begin{equation}\label{energy1stp}
	\begin{split}
	&\int_{\Omega}(v_{k+1}-v_{k})\cdot v_{k+1}
	+h\nu \int_{\Omega}|\nabla v_{k+1}|^{2}
	+\int_{\Omega}(M_{k+1}-M_{k})\cdot \bigl(-\mbox{div}(\xi(\phi_{k})\nabla M_{k+1})\\
	&+\frac{\xi(\phi_{k})}{\alpha^{2}}(|M_{k+1}|^{2}M_{k+1}-M_{k})\bigr)+h\int_{\Omega}\Bigl|\mbox{div}(\xi(\phi_{k})\nabla M_{k+1})-\frac{\xi(\phi_{k})}{\alpha^{2}}(|M_{k+1}|^{2}M_{k+1}-M_{k})\Bigr|^{2}\\
	&\displaystyle +\frac{1}{2}\int_{\Omega}\bigl(\xi(\phi_{k+1})-\xi(\phi_{k})\bigr)|\nabla M_{k+1}|^{2}+\frac{1}{4\alpha^{2}}\int_{\Omega}\bigl(\xi(\phi_{k+1})-\xi(\phi_{k})\bigr)(|M_{k+1}|^{2}-1)^{2}\\
	&+\eta\int_{\Omega}\nabla\phi_{k+1}(\nabla\phi_{k+1}-\nabla\phi_{k})+\frac{1}{\eta}\int_{\Omega}(\phi_{k+1}-\phi_{k})(\phi^{3}_{k+1}-\phi_{k})+h\int_{\Omega}|\nabla \mu_{k+1}|^{2}=0,
	\end{split}
	\end{equation}
	One now performs integration by parts in the last term of \eqref{energy1stp}$_1$ and uses the following identity
	\begin{equation}\label{algebricidnty}
	\begin{split}
	\ A\cdot(A-B)=\frac{|A|^{2}}{2}-\frac{|B|^{2}}{2}+\frac{|A-B|^{2}}{2}\,\,\text{ for all }\,\,A,\,B\in\mathbb{R}^{m},\,m\in\eN,
	\end{split}
	\end{equation}
	on the results of the previous manipulations and the first term of \eqref{energy1stp}$_1$
	and Lemma~\ref{algebriclem} on terms involving $|M_{k+1}|^{2}M_{k+1}-M_{k}$, $\phi^{3}_{k+1}-\phi_{k}$ respectively,
	to infer
	\begin{equation}\label{energy2stp}
	\begin{split}
	&\frac{1}{2}\int_{\Omega}|v_{k+1}|^{2}-\frac{1}{2}\int_{\Omega}|v_{k}|^{2}+\frac{1}{2}\int_{\Omega}|v_{k+1}-v_{k}|^{2}+h\nu\int_{\Omega}|\nabla v_{k+1}|^{2}+\frac{1}{2}\int_{\Omega}\xi(\phi_{k+1})|\nabla M_{k+1}|^{2}\\
	&-\frac{1}{2}\int_{\Omega}\xi(\phi_{k})|\nabla M_{k}|^{2}+\frac{1}{2}\int_{\Omega}\xi(\phi_{k})|\nabla M_{k+1}-\nabla M_{k}|^{2}+\frac{1}{4\alpha^{2}}\int_{\Omega}\xi(\phi_{k+1})\bigl(|M_{k+1}|^{2}-1\bigr)^{2}\\
	&-\frac{1}{4\alpha^{2}}\int_{\Omega}\xi(\phi_{k})\bigl(|M_{k}|^{2}-1\bigr)^{2}
	+\frac{1}{4\alpha^{2}}\int_{\Omega}\xi(\phi_{k})\bigl(|M_{k+1}|^{2}-|M_{k}|^{2}\bigr)^{2}\\
	&+\frac{1}{2\alpha^{2}}\int_{\Omega}\xi(\phi_{k})|M_{k+1}\cdot(M_{k+1}-M_{k})|^{2}
	+\frac{1}{2\alpha^{2}}\int_{\Omega}\xi(\phi_{k})|M_{k+1}-M_{k}|^{2}\\
	& +h\int_{\Omega}\Bigl|\mbox{div}(\xi(\phi_{k})\nabla M_{k+1})-\frac{\xi(\phi_{k})}{\alpha^{2}}(|M_{k+1}|^{2}M_{k+1}-M_{k})\Bigr|^{2}+\frac{\eta}{2}\int_{\Omega}|\nabla\phi_{k+1}|^{2}-\frac{\eta}{2}\int_{\Omega}|\nabla\phi_{k}|^{2}\\
	&+\frac{\eta}{2}\int_{\Omega}|\nabla\phi_{k+1}-\nabla\phi_{k}|^{2}+\frac{1}{4\eta}\int_{\Omega}(\phi_{k+1}^{2}-1)^{2}-\frac{1}{4\eta}\int_{\Omega}(\phi_{k}^{2}-1)^{2}+\frac{1}{4\eta}\int_{\Omega}(\phi^{2}_{k+1}-\phi^{2}_{k})^2\\
	&+\frac{1}{2\eta}\int_{\Omega}(\phi^{2}_{k+1}-\phi_{k+1}\phi_{k})^{2}
	+\frac{1}{2\eta}\int_{\Omega}(\phi_{k+1}-\phi_{k})^{2}+h\int_{\Omega}|\nabla \mu_{k+1}|^{2}\leqslant0.
	\end{split}
	\end{equation}
	Ignoring some of the positive terms appearing on the left hand side of \eqref{energy2stp}, one at once concludes the discrete energy estimate \eqref{discreteestimate} from \eqref{energy2stp}.  
	
	\subsection{Proof of the existence of weak solutions to \texorpdfstring{\eqref{timediscretesystem}}{someref}}\label{existencetimediscrete}
	In order to prove the existence of a weak solution to the time discrete problem \eqref{timediscretesystem}, we will use Leray-Schauder fixed point principle. For this purpose we consider the following Hilbert spaces 
	\begin{equation*}
	\begin{split}
	& X=\WND\times W^{2,2}_{n}(\Omega)\times W^{1,2}(\Omega)\times W^{1,2}(\Omega),\\
	& Y= \bigl(\WND\bigr)'\times L^2(\Omega)\times\bigl(W^{1,2}(\Omega)\bigr)'\times \bigl(W^{1,2}(\Omega)\bigr),
	\end{split}
	\end{equation*}
	with norms defined as a sum of norms on each factor of the corresponding cartesian product, and operators $\mathcal{N}_{k},\mathcal{F}_{k}:X\to Y$. For $w=(v,M,\phi,\mu)\in X,$ we define the operator $\mathcal{N}_{k}$ as follows
	\begin{equation}\label{Lk}
	\begin{split}
	\mathcal{N}_{k}(w)=\begin{pmatrix}
	\mathcal{A}v\\[2.mm]
	-\dvr(\xi(\phi_{k})\nabla M)+\frac{\xi(\phi_{k})}{\alpha^{2}}(|M|^{2}M-M_{k})+\overline{M}\\[2.mm]
	-\eta\Delta_N\phi+\frac{1}{\eta}(\phi^{3}-\phi_{k})+\overline{\phi}\\
	-\Delta_{N}\mu+\overline{\mu}
	\end{pmatrix}
	\end{split}
	\end{equation}
	where $\displaystyle\overline{f}=\int_\Omega f$ for any $f\in L^1(\Omega)$, and $\mathcal{A}: \WND\to \bigl(\WND\bigr)'$ and $-\Delta_N: W^{1,2}(\Omega)\to \bigl(W^{1,2}(\Omega)\bigr)'$ are given for all $(v,\varphi)\in \WND\times W^{1,2}(\Omega)$ by
	\begin{align*}
	\langle\mathcal{A}v,\widetilde\psi_1\rangle&=\nu\int_{\Omega}\nabla v\cdot\nabla \widetilde\psi_1\text{ for all }\widetilde\psi_1\in\WND,\\
	\langle-\Delta_N\varphi,\widetilde\psi_3\rangle&=\int_{\Omega}\nabla\varphi\cdot\nabla\widetilde\psi_3\text{ for all }\widetilde\psi_3\in W^{1,2}(\Omega).
	\end{align*}
	We further introduce $\mathcal{F}_{k}$ for $w=(v,M,\phi,\mu)\in X,$ as follows
	\begin{equation}\label{Fk}
	\begin{array}{ll}
	\mathcal{F}_{k}(w)=\begin{pmatrix}
	&\ -\frac{v-v_{k}}{h}-(v\cdot\nabla)v-\nabla\mu\phi_{k}+\frac{\xi(\phi_{k})}{\alpha^{2}}(|M|^{2}M-M_{k})\nabla M-\dvr(\xi(\phi_{k})\nabla M)\nabla M\\[4.mm]
	&\ -\frac{M-M_{k}}{h}-(v\cdot\nabla)M+\overline{M}\\[4.mm]
	& \ \mu- H_{0}(\phi,\phi_{k})\frac{|\nabla M|^{2}}{2}-\frac{H_{0}(\phi,\phi_{k})}{4\alpha^{2}}(|M|^{2}-1)^{2}+\overline{\phi}\\[4.mm]
	& \  -\frac{\phi-\phi_{k}}{h}-(v\cdot\nabla)\phi_{k}+\overline\mu
	\end{pmatrix}.
	\end{array}
	\end{equation}
	One observes that $w_{k+1}=(v_{k+1},M_{k+1},\phi_{k+1},\mu_{k+1})\in X$ is a weak solution to the time discrete problem \eqref{timediscretesystem} iff the following holds
	$$\mathcal{N}_{k}(w_{k+1})=\mathcal{F}_k(w_{k+1})\,\,\mbox{in}\,\, Y.$$
	The existence of such $w_{k+1}\in X$ will be proved via the Leray-Schauder fixed point theorem. As a prerequisite we need to show the following:
	\begin{itemize}
		\item The operator $\mathcal{N}_{k}: X\to Y$ is bounded and invertible and the inverse operator $\mathcal{N}^{-1}_{k}: Y\to X$ is bounded and continuous.\\
		\item The operator $\mathcal{F}_{k}: X\to Y$ is continuous and compact.
	\end{itemize}
	The first item is proved in Section \ref{Nk} while the assertion of the second item is obtained in Section \ref{FkProp}.  
	\subsubsection{The operator $\mathcal{N}_{k}: X\to Y$ is bounded and invertible and the inverse operator $\mathcal{N}^{-1}_{k}: Y\to X$ is bounded and continuous.}\label{Nk} Showing the boundedness of $\mathcal{N}_k:X\to Y$ is straightforward therefore we will not deal with details and we focus on the invertibility of $\mathcal{N}_k$ and properties of the inverse. To this end we will proceed via defining  another operator $\mathcal{S}_{k}$ retaining a similar structure as that of $\mathcal{N}_{k}$ but defined on a space $W\supset X.$\\
	 Denoting $$W=\WND\times W^{1,2}(\Omega)\times W^{1,2}(\Omega)\times W^{1,2}(\Omega)$$ we define the operator $\mathcal{S}_k:W\to W'$ as
	\begin{equation*}
	\begin{split}
	\mathcal{S}_{k}(w)=\begin{pmatrix}
	\ \mathcal{A}v\\[2.mm]
	-\dvr_N(\xi(\phi_{k})\nabla M)+\frac{\xi(\phi_{k})}{\alpha^{2}}(|M|^{2}M-M_{k})+\overline{M}\\[2.mm]
	-\eta\Delta_N\phi+\frac{1}{\eta}(\phi^{3}-\phi_{k})+\overline{\phi}\\
	-\Delta_{N}\mu+\overline{\mu}
	\end{pmatrix},
	\end{split}
	\end{equation*}
	where $-\dvr_N:L^2(\Omega)\to (W^{1,2}(\Omega))^{'}$ is defined for $\Phi\in L^2(\Omega)$ as 	
	\begin{equation*}
	\langle-\dvr_N \Phi,\psi_2\rangle=\int_\Omega \Phi\cdot\nabla\psi_2\text{ for all }\psi_2\in W^{1,2}(\Omega).
	\end{equation*}
	The goal is to show that $\mathcal{S}_k$ is invertible and the restriction of the inverse operator $\mathcal{S}_k^{-1}$ is bounded and continuous from $Y$ to $X.$ Then by the coincidence of $\mathcal{N}_k$ and the restriction of $\mathcal{S}_k$ on $X$ the boundedness and continuity of $\mathcal{N}_k^{-1}$ follows. 
	
	\enlargethispage{-\baselineskip}
	
	\textbf{$\mathcal{S}_k^{-1}:W'\to W$ is bounded and continuous:} We begin with the justification of the strong monotonicity of $\mathcal{S}_k$ on $W$. To this end, we consider an arbitrary couple $w_1,w_2\in W$, $w_i=(v_i,M_i,\phi_i,\mu_i),\ i=1,2$ and we compute
	\begin{align*}
	\langle\mathcal{S}_k(w_1)-\mathcal{S}_k(w_2),w_1-w_2\rangle =& \int_\Omega |\nabla(v_1-v_2)|^2+\int_\Omega\xi(\phi_k)|\nabla (M_1-M_2)|^2\\&+\int_\Omega\frac{\xi(\phi_k)}{\alpha^2}\bigl(|M_1|^2M_1-|M_2|^2M_2\bigr)\cdot(M_1-M_2)+\bigl|\overline{M_1-M_2}\bigr|^2\\
	& +\eta\int_\Omega|\nabla (\phi_1-\phi_2)|^2+\frac{1}{\eta}\int_\Omega\bigl(\phi_1^3-\phi_2^3\bigr)(\phi_1-\phi_2)+\bigl(\overline{\phi_1-\phi_2}\bigr)^2\\
	&+\int_\Omega|\nabla(\mu_1-\mu_2)|^2+\bigl(\overline{\mu_1-\mu_2}\bigr)^2=\sum_{i=1}^9 I_i.
	\end{align*}
	We show that the latter sum bounds from above the square of the norm in $W$ of the difference $w_1-w_2$. We infer $I_2\geq c_1\|\nabla (M_1-M_2)\|^2_{L^2(\Omega)}$ using the lower bound on $\xi$ and the monotonicity of the mapping $\alpha\mapsto |\alpha|^2\alpha$ implies $I_3,I_6\geq 0$. Employing the lower bound on $\xi$ again and Lemma~\ref{Lem:PoincareAverage} we infer $I_2+I_4\geq c\|M_1-M_2\|^2_{W^{1,2}(\Omega)}$,  $I_5+I_7\geq c\|\phi_1-\phi_2\|^2_{W^{1,2}(\Omega)}$ and $I_8+I_9\geq c\|\mu_1-\mu_2\|^2_{W^{1,2}(\Omega)}$. Hence we conclude that 
	\begin{equation*}
	\langle\mathcal{S}_k(w_1)-\mathcal{S}_k(w_2),w_1-w_2\rangle\geq c\|w_1-w_2\|^2_W,
	\end{equation*}
	i.e., $\mathcal{S}_k$ is strongly monotone.
	
	With the help of the embedding $W^{1,2}(\Omega)\hookrightarrow L^6(\Omega)$ we justify the boundedness of $\mathcal{S}_k$. Applying the Lebesgue dominated convergence theorem it can be immediately checked that $\mathcal{S}_k$ is radially continuous on $W$, i.e., for each pair $w,\tilde w\in W$ the function $t\in\eR\mapsto \langle\mathcal{S}_k(w+t\tilde{w}),\tilde w\rangle$ is continuous. Repeating the arguments used for the justification of the strong monotonicity of $\mathcal{S}_k$ we conclude that $\langle \mathcal{S}_k(w),w\rangle \geq c\|w\|^2_W-c_k$, for any $w\in W$ with $c_k$ depending on $\phi_k$ and $M_k$. The latter inequality implies that $\mathcal{S}_k$ is coercive on $W$, i.e.,
	$$\lim_{\|w\|_W\to\infty}\frac{\langle\mathcal{S}_k(w),w\rangle}{\|w\|_W}=\infty.$$ 
	The application of \cite[Theorem 2.14]{Ro05} yields the existence of the inverse operator $\mathcal{S}_k^{-1}:W'\to W$ that is bounded and Lipschitz continuous.\\[2.mm]
	
	 \textbf{$\mathcal{S}^{-1}_k$ is bounded and continuous from $Y$ to $X$:} We first verify the boundedness and to this end it is sufficient to focus on the second component of the operator under consideration. We observe that for fixed $F\in Y,$ $M$ is a weak solution to 
	\begin{equation}\label{Mkp2}
	\begin{alignedat}{2}
	\Delta M=&\frac{1}{\xi(\phi_k)}\Bigl(F_2-\xi'(\phi_{k})\nabla M\nabla\phi_k+\frac{\xi(\phi_{k})}{\alpha^{2}}\bigl(|M|^{2}M-M_{k}\bigr)\Bigr)\quad&&\text{ in }\Omega,\\
	\partial_{n}M=&0&&\text{ on }\partial\Omega,\\
	\end{alignedat}
	\end{equation}
	where $F_2\in L^2(\Omega)$ is the second component of $F$. Using the bounds on $\xi$ from \eqref{XiAssum} again, the facts that $M\in W^{1,2}(\Omega)$ and $\nabla\phi_k\in W^{1,2}(\Omega)\hookrightarrow L^6(\Omega)$ we get $\nabla M\nabla\phi_k\in L^\frac{3}{2}(\Omega)$, $|M|^2M\in L^2(\Omega)$ and consequently
	\begin{equation}\label{Mkp3}
	\begin{split}
	\Delta M=&\tilde f\in L^q(\Omega),\\
	\partial_{n}M=&0\text{ on }\partial\Omega
	\end{split}
	\end{equation}
	with $q=\frac{3}{2}$.
	One now applies Lemma~\ref{Lem:PoissReg}, recalling the assumption $\Omega$ is of class $C^{1,1}$, to conclude that
	$M\in W^{2,q}_{n}(\Omega)$ and
	\begin{equation}\label{MSecDerBound}
	\|M\|_{W^{2,q}(\Omega)}\leq c(\|\tilde f\|_{L^q(\Omega)}+\|M\|_{W^{1,q}(\Omega)})\leq C_k\left(1+\|F_2\|_{L^q(\Omega)}+\|M\|^3_{L^2(\Omega)}+\|\phi_k\|_{W^{2,2}(\Omega)}\|M\|_{W^{1,s}(\Omega)}\right)
	\end{equation} 
	with $s=2$.	Then we get $\nabla M\in W^{1,\frac{3}{2}}(\Omega)\hookrightarrow L^3(\Omega)$, which implies that we have \eqref{Mkp3} with $q=2$. Applying Lemma~\ref{Lem:PoissReg} again one concludes $M\in W^{2,2}_n(\Omega)$ and bound \eqref{MSecDerBound} with $q=2$ and $s=3$. We estimate in this case $\|M\|_{W^{1,3}(\Omega)}$ on the right hand side of \eqref{MSecDerBound} using the embedding $W^{2,\frac{3}{2}}(\Omega)\hookrightarrow W^{1,3}(\Omega)$ and \eqref{MSecDerBound} with $q=\frac{3}{2}$, $s=2$ to conclude that the restriction of $\mathcal{S}^{-1}_k$ on $Y$ maps bounded sets in $Y$ to bounded sets in $X$.\\[2.mm]
	 In order to show the continuity of $\mathcal{S}^{-1}_k:Y\to X$, we consider a sequence $\{F_j\}\subset Y$ such that $F_j\to F$ in $Y$. Then thanks to the continuity of $\mathcal{S}_k^{-1}$ from $Y$ to $W$ we have for $M_j$, $M$ corresponding to $F_j$, $F$ respectively, that $M_j\to M$ in $W^{1,2}(\Omega)$. Moreover, from \eqref{Mkp2} we observe that $M_j-M$ is a weak solution to     \begin{equation}\label{MDifEq}
	\begin{alignedat}{2}
	\Delta (M_j-M)=&\frac{1}{\xi(\phi_k)}\Bigl((F_j-F)_2-\xi'(\phi_{k})\nabla (M_j-M)\nabla\phi_k+\frac{\xi(\phi_{k})}{\alpha^{2}}\bigl(|M_j|^{2}M_j-|M|^2M\bigr)\Bigr)&&\text{ in }\Omega,\\
	\partial_{n}(M_j-M)=&0&&\text{ on }\partial\Omega,
	\end{alignedat}
	\end{equation}
	where $(F_j-F)_2$ stands for the second component of $F_j-F$. By Lemma~\ref{Lem:PoissReg} we obtain similarly as above
	\begin{align*}
	\|M_j-M\|_{W^{2,\frac{3}{2}}(\Omega)}\leq & c \bigl(\|(F_j-F)_2\|_{L^\frac{3}{2}(\Omega)}+\|\nabla (M_j-M)\|_{L^2(\Omega)}\|\nabla \phi_k\|_{L^6(\Omega)}\\&+(\|M_j\|^2_{L^6(\Omega)}+\|M\|^2_{L^6(\Omega)})\|M_j-M\|_{L^6(\Omega)}+\|M_j-M\|_{W^{1,\frac{3}{2}}(\Omega)}\bigr).
	\end{align*}
	Hence the convergence $F_j\to F$ in $Y$ and the continuity of $\mathcal{S}_k^{-1}$ from $Y$ to $W$ implies $M_j \to M$ in $W^{2,\frac{3}{2}}(\Omega)$, $M_j\to M$ in $W^{1,3}(\Omega)$ consequently. Going back to \eqref{MDifEq} we infer
	\begin{align*}
	\|M_j-M\|_{W^{2,2}(\Omega)}\leq & c \bigl(\|(F_j-F)_2\|_{L^2(\Omega)}+\|\nabla (M_j-M)\|_{L^3(\Omega)}\|\nabla \phi_k\|_{L^6(\Omega)}\\&+(\|M_j\|^2_{L^6(\Omega)}+\|M\|^2_{L^6(\Omega)})\|M_j-M\|_{L^6(\Omega)}+\|M_j-M\|_{W^{1,2}(\Omega)}\bigr)
	\end{align*}
	and subsequently the convergence $M_j\to M$ in $W^{2,2}(\Omega)$, which concludes the continuity of $\mathcal{S}^{-1}_k:Y\to X$. Then  we immediately conclude that the restriction of $\mathcal{S}_k$ on $X$ coincides with $\mathcal{N}_k$ as the amount of regularity for $M$ allows for performing the integration by parts in the second component of $\mathcal{S}_k$. Then the restriction of $\mathcal{S}_k^{-1}$ on $Y$ coincides with $\mathcal{N}_k^{-1}$, which in particular implies the boundedness and continuity of the latter operator. 
	\subsubsection{The operator $\mathcal{F}_{k}: X\to Y$ is continuous and compact}\label{FkProp}  Now we show that $\mathcal{F}_{k}:X\to Y$ is compact. Let us introduce the space
	$$Z=L^{\frac{3}{2}}(\Omega)\times W^{1,\frac{3}{2}}(\Omega)\times L^3(\Omega)\times L^{3}(\Omega).$$ 
	To prove the compactness of $\mathcal{F}_{k}:X\to Y,$ first it is shown that $\mathcal{F}_{k}:X\to Z$ is a bounded operator and then we show that $\mathcal{F}_{k}:X\to Z$ is continuous. Finally, in view of the compact embedding of $Z$ into $Y$ and the linearity of the inclusion map we conclude that $\mathcal{F}_{k}: X\to Y$ is continuous and compact.
	
	\enlargethispage{-\baselineskip}
	
	\textbf{Boundedness of $\mathcal{F}_{k}:X\to Z:$} This is going to be a consequence of the following estimates whose obtainment are explained after \eqref{boundednes}: 
	\begin{equation}\label{boundednes}
	\begin{split}
	\|(v\cdot\nabla)v\|_{L^{\frac{3}{2}}(\Omega)}&\leqslant C\|v\|^{2}_{W^{1,2}(\Omega)},\\
	\|\nabla\mu\phi_{k}\|_{L^{\frac{3}{2}}(\Omega)}&\leqslant C_{k}\|\mu\|_{W^{1,2}(\Omega)},\\
	\|\xi(\phi_{k})(|M|^{2}M-M_{k})\nabla M\|_{L^{\frac{3}{2}}(\Omega)}&\leqslant C_{k}(\|M\|^{4}_{W^{2,2}(\Omega)}+\|M\|_{W^{2,2}(\Omega)}),\\
	\|\mathrm{div}(\xi(\phi_{k})\nabla M)\nabla M\|_{L^{\frac{3}{2}}(\Omega)}&\leqslant C_{k}\|M\|^{2}_{W^{2,2}(\Omega)},\\
	\|(v\cdot\nabla)M\|_{W^{1,\frac{3}{2}}(\Omega)}&\leqslant C\|v\|_{W^{1,2}(\Omega)}\|M\|_{W^{2,2}(\Omega)},\\
	\frac{1}{2}\|H_{0}(\phi,\phi_{k})|\nabla M|^{2}\|_{L^{3}(\Omega)}&\leqslant C_{k}\|M\|^{2}_{W^{2,2}(\Omega)},\\[3.mm]
	\frac{1}{4\alpha^2}\|H_{0}(\phi,\phi_{k})(|M|^{2}-1)^{2}\|_{L^{3}(\Omega)}&\leqslant C_{k}(\|M\|^{4}_{W^{2,2}(\Omega)}+1),\\[3.mm]
	\|(v\cdot\nabla)\phi_{k}\|_{L^{3}(\Omega)}& \leqslant C_{k}\|v\|_{W^{1,2}(\Omega)}
	\end{split}
	\end{equation}
	and the linear terms are obviously bounded in corresponding spaces. We will explain the obtainment of the bounds \eqref{boundednes}$_{3}$, \eqref{boundednes}$_{4}$, \eqref{boundednes}$_{5}$ and \eqref{boundednes}$_{6}$. The other boundedness estimates are relatively easy to deal with.\\
	To show \eqref{boundednes}$_{3}$, we use the continuous embedding of $W^{2,2}_{n}(\Omega)$,  $W^{1,2}(\Omega)$ respectively in $L^{\infty}(\Omega)$,  $L^{6}(\Omega)$ and \eqref{XiAssum}$_{2}$, and furnish the bound of $|M|^{2}M\nabla M$ in   $L^{\infty}(\Omega)\times L^{6}(\Omega)\hookrightarrow L^{6}(\Omega)\hookrightarrow L^{\frac{3}{2}}(\Omega).$\\[2.mm]
	Now we will prove \eqref{boundednes}$_{4}.$
	In that direction one expands $\mathrm{div}(\xi(\phi_{k})\nabla M)\nabla M$ to have:
	\begin{equation*}
	\begin{array}{l}
	\mathrm{div}(\xi(\phi_{k})\nabla M)=\xi(\phi_{k})\Delta M+\xi^{'}(\phi_{k})\nabla M\nabla\phi_{k}.
	\end{array}
	\end{equation*}
	We estimate the terms on the right hand side of the latter identity. Since $\phi_{k}$ is bounded in $W^{2,2}_{n}(\Omega),$ one uses \eqref{XiAssum}$_{3}$ and the embedding $W^{1,2}(\Omega)\times W^{1,2}(\Omega)\hookrightarrow L^{6}(\Omega)\times L^{6}(\Omega)\hookrightarrow L^{3}(\Omega)$ to furnish the $L^{3}(\Omega)$ norm bound for the second summand. In view of \eqref{XiAssum}$_{2},$ the $L^{2}(\Omega)$ norm of the first summand is bounded by $\|M\|_{W^{2,2}(\Omega)}.$ This altogether provides that $$\|\mathrm{div}(\xi(\phi_{k})\nabla M)\|_{L^{2}(\Omega)}\leqslant C_{k}\|M\|_{W^{2,2}(\Omega)}.$$ Further one uses the embedding $L^{2}(\Omega)\times W^{1,2}(\Omega)\hookrightarrow L^{2}(\Omega)\times L^{6}(\Omega)\hookrightarrow L^{\frac{3}{2}}(\Omega)$ to have
	$$\ \|\mathrm{div}(\xi(\phi_{k})\nabla M)\nabla M\|_{L^{\frac{3}{2}}(\Omega)}\leqslant\|\mathrm{div}(\xi(\phi_{k})\nabla M)\|_{L^{2}(\Omega)}\|\nabla M\|_{L^{6}(\Omega)}\leqslant C_{k}\|M\|^{2}_{W^{2,2}(\Omega)},$$
	which is \eqref{boundednes}$_{4}$.\\
	Now we will show \eqref{boundednes}$_{5}.$ The boundedness of $(v\cdot\nabla)M$ in $L^{\frac{3}{2}}(\Omega)$ is clear from the continuous embedding $W^{1,2}(\Omega)\hookrightarrow L^{6}(\Omega)$ and $L^{6}(\Omega)\times L^{6}(\Omega)\hookrightarrow L^{3}(\Omega)\hookrightarrow L^{\frac{3}{2}}(\Omega).$ In order to estimate the spatial derivatives of $(v\cdot\nabla)M,$ one needs to estimate terms of the form $(v)_{i}\partial_{ij}M$ and $\partial_{j}(v)_{i}\partial_{i}M.$ Since $W^{1,2}(\Omega)\hookrightarrow L^{6}(\Omega)$ and $L^{6}(\Omega)\times L^{2}(\Omega)\hookrightarrow L^{\frac{3}{2}}(\Omega)$ the boundedness of $\|(v)_{i}\partial_{ij}M\|_{L^{\frac{3}{2}}(\Omega)}$ follows. Similarly one uses $W^{2,2}(\Omega)\hookrightarrow W^{1,6}(\Omega)$ to furnish the bound on $\|\partial_{j}(v)_{i}\partial_{i}M\|_{L^{\frac{3}{2}}(\Omega)}.$ This concludes \eqref{boundednes}$_{5}$.\\
	Finally, let us show \eqref{boundednes}$_{6}.$
	One recalls the definition of $H_{0}(\cdot,\cdot)$ from \eqref{H0}. Since $\xi(\cdot)\in C^{1}(\mathbb{R}),$ one uses the mean value theorem and the upper bound of $\xi'(\cdot)$, cf.\ assumption \eqref{XiAssum} to have $\|H_{0}(\phi,\phi_{k})\|_{L^{\infty}(\Omega)}\leqslant C_{k}$, for some positive constant $C_{k}.$ Further using $W^{2,2}(\Omega)\hookrightarrow W^{1,6}(\Omega)$ one has $\||\nabla M|^2\|_{L^{3}(\Omega)}\leqslant\|\nabla M\|^2_{L^{6}(\Omega)}\leqslant C\|M\|^{2}_{W^{2,2}(\Omega)}.$ Consequently \eqref{boundednes}$_{6}$ follows.\\[3.mm]
	\textbf{The operator $\mathcal{F}_k:X\to Z$ is continuous:} To show this assertion we consider an arbitrary sequence $\{w_j\}\subset X$ and $w\in X$ such that $w_j\to w$ in $X$ as $j\to\infty$. We prove that $\{w_j\}$ possesses a subsequnece $\{w_{j'}\}$ for which $\mathcal{F}_k(w_{j'})\to \mathcal{F}_{k}(w)$ in $Z$. Obviously, if $\mathcal{F}_k$ were not continuous, we would find $\{ \tilde w_j\}$ converging to $\tilde w$ in $X$ with $\|\mathcal{F}_k(\tilde w_j)-\mathcal{F}_k(\tilde w)\|_Z\geq m>0$ for all $j$, which contradicts the existence of a subsequence $\{\tilde w_ {j'}\}$ such that $\mathcal{F}_k(\tilde w_{j'})\to \mathcal{F}_k(\tilde w)$. Let us fix $\{w_ j\}$ converging to $w$ in $X$. We select a subsequence $\{w_{j'}\}$ such that each component of $\{w_{j'}\}$, in particular $\{M_{j'}\},\{\nabla M_ {j'}\}, \{\phi_ {j'}\}$ converge also a.e.\ in $\Omega$ to the corresponding component of $w$. The passage $j'\to\infty$ in linear terms of $\mathcal{F}_k(w_{j'})$ follows due to their boundedness. We explain in details the passage $j'\to\infty$ in more involved nonlinear terms. Namely, due to the a.e.\ convergence of $\{\phi_{j'}\}$ and $\{\nabla M_ {j'}\}$ and assumption \eqref{XiAssum} we obtain that $H_0(\phi_{j'},\phi_k)|\nabla M_{j'}|^2\to H_0(\phi,\phi_k)|\nabla M|^2$ a.e.\ in $\Omega$ and the convergence in $L^3(\Omega)$ then follows by Lemma~\ref{Lem:GenDomConv} and the embedding $W^{2,2}(\Omega)\hookrightarrow W^{1,6}(\Omega)$. One obtains similarly that due to the a.e convergence of $\{\phi_{j'}\}$, $\{M_{j'}\}$ and the embedding $W^{2,2}(\Omega)\hookrightarrow L^{12}(\Omega)$ the sequence $\{H_0(\phi_{j'},\phi_k)(|M_{j'}|^2-1)^2\}$ converges to $H_0(\phi,\phi_k)(|M|^2-1)^2$ in $L^3(\Omega)$. For the limit passage in the remaining nonlinear terms we employ analogous arguments to that ones used for showing their boundedness. Hence we obtain that $\mathcal{F}_k(w_{j'})\to \mathcal{F}_k(w)$ as $j'\to\infty$. This proves the continuity of the map $\mathcal{F}_{k}:X\rightarrow Z.$\\[2.mm]
	We now observe that $Z$ is compactly embedded in $Y.$ Hence from the boundedness and continuity of $\mathcal{F}_{k}:X\to Z$, and the linearity of the inclusion $Z$ into $Y,$ the compactness and continuity of the map $\mathcal{F}_{k}:X\to Y$ follows.\\[2.mm]
	Finally, one aims to show the existence of a $w_{k+1}\in X$ satisfying
	\begin{equation}\label{fxdpntmap}
	\begin{array}{l}
	\mathcal{N}_{k}(w_{k+1})=\mathcal{F}_{k}(w_{k+1})\,\,\mbox{in}\,\, Y.
	\end{array}
	\end{equation}
	In fact, it suffices to show the existence of a fixed point of the operator $\mathcal{F}_{k}\circ \mathcal{N}^{-1}_{k}$ on $Y$, i.e., the existence of $h_{k+1}\in Y$ satisfying
	\begin{equation}\label{2ndfxd}
	\begin{array}{l}
	h_{k+1}=(\mathcal{F}_{k}\circ \mathcal{N}^{-1}_{k})h_{k+1}\,\,\mbox{in}\,\,Y,
	\end{array}
	\end{equation}
	since from the invertibility of $\mathcal{N}_{k}:X\to Y,$ one can obtain $w_{k+1}\in X$  satisfying \eqref{fxdpntmap} by using $w_{k+1}=\mathcal{N}^{-1}_{k}(h_{k+1})$.\\
	To prove the existence of a fixed point of \eqref{2ndfxd} we apply the Leray-Schauder fixed point theorem \cite[Theorem 10.3]{GilTr01} to the compact and continuous operator $\mathcal{F}_{k}\circ \mathcal{N}^{-1}_{k}$. To this end we verify that:
	\begin{equation}\label{lerayfxpt}
	\begin{array}{l}
	\mbox{There exists}\,\, r>0 \,\,\mbox{such that if}\,\, h\in Y\,\, \mbox{solves}\,\, h=\lambda(\mathcal{F}_{k}\circ \mathcal{N}^{-1}_{k})h\,\,\mbox{with}\,\,\lambda\in[0,1],\\
	\mbox{then it holds}\,\,\|h\|_{Y}\leqslant r.
	\end{array}
	\end{equation}
	Let $h\in Y$ satisfy $h=\lambda(\mathcal{F}_{k}\circ \mathcal{N}^{-1}_{k})h$ in $Y$ with some $\lambda\in [0,1]$. Then
	$$w=(v,M,\phi,\mu)=\mathcal{N}_{k}^{-1}h,$$
	fulfills
	\begin{equation}\label{LkFk}
	\begin{array}{l}
	\mathcal{N}_{k}(w)-\lambda\mathcal{F}_{k}(w)=0\,\,\mbox{in}\,\, Y.
	\end{array}
	\end{equation}
	We first show that
	\begin{equation}\label{boundw}
	\|w\|_{X}\leqslant C_{k}
	\end{equation}
	with $C_k$ independent of $\lambda\in[0,1]$, from which \eqref{lerayfxpt} follows due to the boundedness of $\mathcal{N}_k$. Therefore we now focus on showing \eqref{boundw}. One can explicitly recall the definitions of $\mathcal{N}_{k}$ and $\mathcal{F}_{k}$ from \eqref{Lk} and \eqref{Fk}, test the first component of \eqref{LkFk} by $v$, the second component by $-\dvr(\xi(\phi_k)\nabla M)+\frac{\xi(\phi_{k})}{\alpha^{2}}(|M|^{2}M-M_{k})$, the third component by $\frac{\phi-\phi_{k}}{h}$ and the fourth component by $\mu$, use \eqref{algebricidnty}, Lemma~\ref{algebriclem} (similarly as we obtained \eqref{energy2stp} from \eqref{energy1stp} during the derivation of the discrete version of the energy inequality in the first part of the proof) and drop some positive terms from the left hand side (exactly as we have obtained \eqref{discreteestimate} from \eqref{energy2stp}) to furnish:
	\begin{equation}\label{energylambda}
	\begin{split}
	&\frac{\lambda}{h}\left(\frac{1}{2}\int_{\Omega}|v|^{2}-\frac{1}{2}\int_{\Omega}|v_{k}|^{2}\right)+\nu\int_{\Omega}|\nabla v|^{2}+\frac{\lambda}{h}\left(\frac{1}{2}\int_{\Omega}\xi(\phi)|\nabla M|^{2}
	-\frac{1}{2}\int_{\Omega}\xi(\phi_{k})|\nabla M_{k}|^{2}\right)\\
	& +\frac{\lambda}{h}\left(\frac{1}{4\alpha^{2}}\int_{\Omega}\xi(\phi)\bigl(|M|^{2}-1\bigr)^{2}-\frac{1}{4\alpha^{2}}\int_{\Omega}\xi(\phi_{k})\bigl(|M_{k}|^{2}-1\bigr)^{2}\right)\\
	& +\int_{\Omega}\left|\mbox{div}(\xi(\phi_{k})\nabla M)-\frac{\xi(\phi_{k})}{\alpha^{2}}(|M|^{2}M-M_{k})\right|^{2}+\frac{1}{h}\left(\frac{\eta}{2}\int_{\Omega}|\nabla\phi|^{2}-\frac{\eta}{2}\int_{\Omega}|\nabla\phi_{k}|^{2}\right)\\
	&+\frac{1}{h}\left(\frac{1}{4\eta}\int_{\Omega}(\phi^{2}-1)^{2}-\frac{1}{4\eta}\int_{\Omega}(\phi_{k}^{2}-1)^{2}\right)+\int_{\Omega}|\nabla \mu|^{2}+(1-\lambda)\overline\mu ^{2}+\frac{1-\lambda}{h}\overline{\phi}^2\\
	&-\frac{1-\lambda}{h}\overline\phi\overline\phi_k +(1-\lambda)\overline M \int_\Omega\left(-\mbox{div}(\xi(\phi_{k})\nabla M)+\frac{\xi(\phi_{k})}{\alpha^{2}}(|M|^{2}M-M_{k})\right)\leqslant0.
	\end{split}
	\end{equation}
	We used the notation $\displaystyle\overline f=\int_\Omega f$ for $f\in L^1(\Omega)$. Let us note that the attainment of the inequality \eqref{energylambda} is independent of the values of $\lambda\in[0,1]$.
	
	One now uses Young's and H\"{o}lder's inequality to estimate the terms appearing on the last line of \eqref{energylambda}:
	\begin{equation}\label{PhiProdEst}
	\begin{split}
	\left|\frac{1-\lambda}{h}\overline\phi\overline\phi_k\right|\leq &\frac{1-\lambda}{h}\left(\gamma|\Omega|^3\int_\Omega|\phi|^4+c\gamma^{-\frac{1}{3}}|\Omega|\int_\Omega|\phi_k|^\frac{4}{3}\right)\\
	\leq &\frac{1-\lambda}{h}\left(3\gamma|\Omega|^3\left(\int_\Omega(\phi^2-1)^2+\frac{1}{2}|\Omega|\right)+c\gamma^{-\frac{1}{3}}|\Omega|\int_\Omega|\phi_k|^\frac{4}{3}\right),
	\end{split}
	\end{equation}
	\begin{equation}\label{Jensen}
	\begin{split}
	&\left|(1-\lambda)\overline{M} \int_{\Omega}\left(-\mbox{div}(\xi(\phi_{k})\nabla M)+\frac{\xi(\phi_{k})}{\alpha^{2}}(|M|^{2}M-M_{k})\right)\right|\\
	&\leqslant (1-\lambda)\epsilon|\Omega|\int_{\Omega}\left|\mbox{div}(\xi(\phi_{k})\nabla M)-\frac{\xi(\phi_{k})}{\alpha^{2}}(|M|^{2}M-M_{k})\right|^{2}+c\frac{1-\lambda}{\epsilon}\overline{M}^{2}
	\end{split}
	\end{equation}
	for some positive parameters $\gamma,\epsilon>0$.\\
	Since $(1-\lambda)$ is non negative and always bounded by one, for small enough choice of the parameters $\gamma,\epsilon>0,$ the first terms on the right hand side of \eqref{PhiProdEst} and \eqref{Jensen} can be absorbed respectively by the seventh and fifth summands appearing on the left hand side of \eqref{energylambda}. One still needs to estimate the second term on the right hand side of \eqref{Jensen}. In that direction we recall from \eqref{LkFk} that $M\in W^{2,2}_n(\Omega)$ solves the following equation:
	\begin{equation}\label{Mlambda}
	\begin{split}
	\lambda\int_{\Omega}\left(\frac{M-M_{k}}{h}+(v\cdot\nabla)M\right)\cdot\widetilde{\psi}_2+(1-\lambda)\overline{M}\overline{\widetilde\psi_2}=\int_{\Omega}-\xi(\phi_{k})\nabla M\cdot\nabla\widetilde\psi_2-\frac{\xi(\phi_{k})}{\alpha^{2}}(|M|^{2}M-M_{k})\cdot\widetilde\psi_2
	\end{split}
	\end{equation}
	for all $\widetilde\psi_2\in W^{1,2}(\Omega)$.
	Setting $\widetilde{\psi}_{2}=M$ in \eqref{Mlambda}, using \eqref{algebricidnty} and the incompressibility of $v,$ we infer:
	\begin{equation}\label{estimateMMk}
	\begin{split}
	&\frac{\lambda}{2h}\left(\int_{\Omega}|M|^{2}-\int_{\Omega}|M_{k}|^{2}+\int_{\Omega}|M-M_{k}|^{2}\right)+\int_{\Omega}\xi(\phi_{k})|\nabla M|^{2}+\int_{\Omega}\frac{\xi(\phi_k)}{\alpha^2}|M|^{4}\\
	&+(1-\lambda)\overline{M}^{2}=\int_{\Omega}\frac{\xi(\phi_{k})}{\alpha^{2}}M\cdot M_{k}.
	\end{split}
	\end{equation}
	Once again one uses Young's inequality with $\delta>0$ to estimate $$\left|\int_{\Omega}M\cdot M_{k}\right|\leqslant \delta\int_{\Omega}|M|^{4}+c\delta^{-\frac{1}{3}}\int_{\Omega}|M_{k}|^{\frac{4}{3}}.$$
	We choose suitably small value of the parameter $\delta$, use \eqref{XiAssum}$_{2}$ and $\lambda\leq 1$ to have in particular the following from \eqref{estimateMMk}:
	\begin{equation}\label{normboundM}
	\ \int_{\Omega}|\nabla M|^{2}+\int_{\Omega} |M|^{4}+(1-\lambda)\overline{M}^{2}\leqslant C_{k},
	\end{equation}
	where $C_{k}>0$ is independent of $\lambda>0.$ For small enough choice of the parameters $\gamma>0,$ $\epsilon>0,$ using \eqref{normboundM}, \eqref{PhiProdEst} and \eqref{Jensen} in \eqref{energylambda} we obtain:
	\begin{equation}\label{energylambda*}
	\begin{split}
	&\frac{\lambda}{h}\left(\frac{1}{2}\int_{\Omega}|v|^{2}-\frac{1}{2}\int_{\Omega}|v_{k}|^{2}\right)+\nu\int_{\Omega}|\nabla v|^{2}+\frac{\lambda}{h}\left(\frac{1}{2}\int_{\Omega}\xi(\phi)|\nabla M|^{2}
	\ -\frac{1}{2}\int_{\Omega}\xi(\phi_{k})|\nabla M_{k}|^{2}\right)\\
	&\ +\frac{\lambda}{h}\left(\frac{1}{4\alpha^{2}}\int_{\Omega}\xi(\phi)\bigl(|M|^{2}-1\bigr)^{2}-\frac{1}{4\alpha^{2}}\int_{\Omega}\xi(\phi_{k})\bigl(|M_{k}|^{2}-1\bigr)^{2}\right)\\
	&\ +\frac{1}{2}\int_{\Omega}\left|\mbox{div}(\xi(\phi_{k})\nabla M)-\frac{\xi(\phi_{k})}{\alpha^{2}}(|M|^{2}M-M_{k})\right|^{2}+\frac{1}{h}\left(\frac{\eta}{2}\int_{\Omega}|\nabla\phi|^{2}-\frac{\eta}{2}\int_{\Omega}|\nabla\phi_{k}|^{2}\right)\\
	&\ +\frac{1}{h}\left(\frac{1}{8\eta}\int_{\Omega}(\phi^{2}-1)^{2}-\frac{1}{4\eta}\int_{\Omega}(\phi_{k}^{2}-1)^{2}\right)+\int_{\Omega}|\nabla \mu|^{2}+(1-\lambda)\overline{\mu}^{2}\leqslant C_{k}.
	\end{split}
	\end{equation}
	One uses \eqref{energylambda*} and Poincar\'e's inequality (since $v$ solves homogeneous Dirichlet boundary condition) to render $\|v\|_{W^{1,2}(\Omega)}+\|\phi\|_{W^{1,2}(\Omega)}\leqslant C_{k}.$  We conclude $\|M\|_{W^{1,2}(\Omega)}\leqslant C_{k}$ independently of $\lambda$ from \eqref{normboundM}. The first term of \eqref{energylambda*}$_{3}$ provides
	\begin{equation}\label{divMinq}
	\left\|\dvr(\xi(\phi_{k})\nabla M)-\frac{\xi(\phi_{k})}{\alpha^{2}}(|M|^{2}M-M_{k})\right\|_{L^{2}(\Omega)}\leqslant C_{k}.
	\end{equation}
	Since $W^{1,2}(\Omega)\hookrightarrow L^{6}(\Omega)$ one has $\|M\|_{L^{6}(\Omega)}\leqslant C_{k}$ and hence in view of \eqref{divMinq} one furnishes
	\begin{equation*}
	\|\dvr(\xi(\phi_{k})\nabla M)\|_{L^{2}(\Omega)}\leqslant C_{k}.
	\end{equation*}
	Since $\phi_{k}\in W^{2,2}(\Omega)$ and $M\in W^{1,2}(\Omega)$, one concludes similarly to the corresponding proof for $M$, cf.\ arguments leading to \eqref{Mkp3}, that first $\|\Delta M\|_{L^\frac{3}{2}(\Omega)}\leq C_k$. Recalling that $\partial_{n}M\mid_{\partial\Omega}=0,$ and $M$ solves an elliptic problem we get $\|M\|_{W^{2,\frac{3}{2}}(\Omega)}\leqslant C_{k}$ that finally by a bootstrap argument implies $\|M\|_{W^{2,2}(\Omega)}\leqslant C_k$.\\
	It follows directly from \eqref{energylambda*} that $\|\nabla\mu\|_{L^{2}(\Omega)}\leqslant C_{k}$. To conclude that $\|\mu\|_{W^{1,2}(\Omega)}\leqslant C_{k}$ by Lemma~\ref{Lem:PoincareAverage} it is sufficient to show that
	\begin{equation}\label{MuAvEst}
	\left|\overline\mu\right|\leq C_k.
	\end{equation}
	One observes that \eqref{MuAvEst} immediately follows from the bound on the last term appearing in the left hand side of \eqref{energylambda*} for $\lambda\in[0,\frac{1}{2})$. For the case $\lambda\in[\frac{1}{2},1]$ we test 
	\begin{equation}\label{PhiEq}
	-\Delta_N \phi+\frac{1}{\eta}(\phi^3-\phi_k)=\lambda\Bigl(\mu-H_0(\phi,\phi_k)\frac{|\nabla M|^2}{2}+\frac{H_0(\phi,\phi_k)}{4\alpha^2}(|M|^2-1)^2\Bigr) +(1-\lambda)\overline\phi
	\end{equation}
	by one and use the bound of $H_{0}(\phi,\phi_k)$ in $L^{\infty}(\Omega)$, $M$ in $W^{2,2}_{n}(\Omega)$ and $\phi$ in $W^{1,2}(\Omega)$ to arrive at \eqref{MuAvEst}. Hence \eqref{MuAvEst} is true independently of the values of $\lambda\in[0,1]$ and consequently the estimate $\|\mu\|_{W^{1,2}(\Omega)}\leqslant C_{k}$ follows.\\
	In conclusion we have proved inequality \eqref{boundw}. Hence there exists  $w_{k+1}=(v_{k+1},M_{k+1},\phi_{k+1},\mu_{k+1}) \in X$ satisfying the identity \eqref{fxdpntmap}. By construction this quadruple also solves the identities \eqref{identity1}--\eqref{identity4}.\\[2.mm]
	In order to show \eqref{regularityk1} one only needs to improve the regularity of $\phi_{k+1}$. We claim that $\phi_{k+1}\in W^{2,2}_{n}(\Omega).$ One recalls that $\phi_{k+1}$ solves an equation of the form \eqref{PhiEq} with $\lambda=1$ and $\partial_{n}\phi_{k+1}=0$ on $\partial\Omega$. Since $w_{k+1}\in X$ and $H_{0}(\phi_{k+1},\phi_{k})\in L^{\infty}(\Omega),$ it is not hard to observe that $\Delta\phi_{k+1}\in L^{2}(\Omega)$ and hence in view of the regularity properties of a solution to the Poisson equation with the homogeneous Neumann boundary condition one bootstraps the regularity to show $\phi_{k+1}\in W^{2,2}_{n}(\Omega).$ 
	 This finishes the proof of Theorem~\ref{existenceweaksoldiscrete}.
	\end{proof}
We now want to state an auxiliary result corresponding to the bound of $\displaystyle\left|\int_{\Omega} \mu\right|$ which will be used in the next section. 
\begin{lem}\label{boundintmu}
	Let \eqref{assumk} hold. Then there is a constant $\displaystyle C=C\left(\int_{\Omega}\phi_{0}\right)>0$ such that the following inequality is true
	\begin{equation*}
	\left|\int_{\Omega}\mu_{k+1}\right|\leqslant C\left(\|\nabla M_{k+1}\|^{2}_{L^{2}(\Omega)}+\|M_{k+1}\|^{4}_{L^{4}(\Omega)}+\|\phi_{k+1}\|^{3}_{L^{3}(\Omega)}+1\right).
	\end{equation*}  
\end{lem}
\begin{proof}
	First we choose $\widetilde\psi_{3}=1$ in the identities \eqref{identity3} and \eqref{identity4}. The identity \eqref{identity3} at once yields
	\begin{equation}\label{intphiconstant}
	\int_\Omega\phi_{k+1}=\int_{\Omega}\phi_{k},
	\end{equation}
	since $$ \int_{\Omega}(v_{k+1}\cdot\nabla)\phi_{k}=\int_{\Omega}\dvr(v_{k+1}\phi_{k})=0$$ by Green's formula.\\
	By iteration \eqref{intphiconstant} also provides
	\begin{equation}\label{equalsphi0}
	\int_{\Omega}\phi_{k+1}=\int_{\Omega}\phi_{k}=\int_{\Omega}\phi_{0}.
	\end{equation}
	Now we will use \eqref{identity4} with $\widetilde{\psi}_{3}=1.$ One observes that
	\begin{equation}\label{somebounds}
	\begin{split}
	\int_{\Omega}|H_{0}(\phi_{k+1},\phi_{k})(|M_{k+1}|^{2}-1)^{2}|&\leqslant c(\|M_{k+1}\|^{4}_{L^{4}(\Omega)}+1),\\
	 \int_{\Omega}|\phi^{3}_{k+1}|&\leqslant c\|\phi_{k+1}\|^{3}_{L^{3}(\Omega)},
	\end{split}
	\end{equation}
	for some positive constant $c>0.$\\
	Hence by using \eqref{equalsphi0} and \eqref{somebounds} in \eqref{identity4} the desired result follows.
\end{proof}
		\section{Proof of Theorem~\ref{Thm:Main}}\label{sec3} 
		From now on and until the end of this section we fix $T>0.$\\
		For the purposes of this section we will consider a strictly increasing sequence formed by the points $0=t_{0}<t_{1}<...<t_{k}<t_{k+1}<...,$ $k\in \mathbb{N}_{0},$ such that for fixed $N\in\eN,$  $h=\frac{1}{N}=t_{k+1}-t_{k}$, for each $k\in\mathbb{N}_{0}$. By the successive application of Theorem~\ref{existenceweaksoldiscrete} we aim to construct a sequence of solutions $\{(v_{k+1},M_{k+1},\phi_{k+1},\mu_{k+1})\},$ $k\in\mathbb{N}_{0}$ to problem \eqref{timediscretesystem}, by assuming $(v_{k},M_{k},\phi_{k})\in \LND\times W^{2,2}_n(\Omega)\times W^{2,2}_n(\Omega).$
     Recalling the regularity of initial data, it is clear that for the first application of Theorem~\ref{existenceweaksoldiscrete} one can not directly use  $(M_{0},\phi_{0})\in W^{1,2}(\Omega)\times W^{1,2}(\Omega)$. In order to overcome this issue, we employ the density of $W^{2,2}_n(\Omega)$ in $W^{1,2}(\Omega)$ that is discussed for the case of $\Omega$ of class $C^{1,1}$ in \cite[Remark 1.2 iii)]{Dro02}. Therefore we consider sequences $\{M^{N}_{0}\}\subset W^{2,2}_{n}(\Omega)$, $\{\phi^{N}_{0}\}\subset W^{2,2}_{n}(\Omega)$ such that 
	\begin{equation}\label{MN0Cnv}
	  M^{N}_{0}\to M_{0}\text{ in }W^{1,2}(\Omega)
	\end{equation}
    and
	\begin{equation}\label{PhiN0Cnv}
	\phi^{N}_{0}\to \phi_{0}\text{ in }W^{1,2}(\Omega)
	\end{equation}
    as $N\to\infty$. We adopt the notations used in \cite{AbDeGa113}, \cite{AbDeGa213} and \cite{AbelsWeber} in order to introduce suitable interpolation functions corresponding to the unknowns. We fix $N\in\eN$, set $h=\frac{1}{N}$ and define  piecewise constant interpolants corresponding to $\{v,M,\phi\}$ on $[-h,\infty)$ and to $\mu$ on $[0,\infty)$ as follows:
	 $$ v^N(t)=v_0,\ M^N(t)=M^N_0,\ \phi^N(t)=\phi^N_0 \qquad\mbox{for}\qquad t\in[-h,0)$$
	 and 
	 $$f^{N}(t)=f_{k+1}\qquad\mbox{for}\qquad t\in[kh,(k+1)h),$$
	 where $f^N(t)$ represents $v^N(t), M^N(t), \phi^N(t), \mu^N(t)$, and $f_{k}$ represents the corresponding $v_{k}, M_{k}, \phi_{k},\mu_{k}$ for $k\in\mathbb{N}_{0}$.\\ 
	 Next, we define a piecewise affine interpolant $\widetilde f^N$ for $k\in\eN_0$ by 
	 \begin{equation}\label{AffInterpolantDef}
	    \widetilde f^N(t)=\frac{(k+1)h-t}{h}f^N(t-h)+\frac{t-kh}{h}f^N(t)\text{ for }t\in[kh,(k+1)h).    
	 \end{equation}
	 We note that for our purposes it is sufficient to construct the sequences of piecewise affine interpolants $\{\widetilde v^N\}$, $\{\widetilde \phi^N\}$ and $\{\widetilde M^N\}$. Let us introduce the notation that is used throughout this section. We denote the shift in time and the difference quotient of a function $f$ as follows
	 \begin{equation}\label{shifttimef}
        \begin{split}
	  f_{h}(t)&=\left(\tau_{-h}f\right)(t)=f(t-h),\\
	 \partial_{t,h}^{-}f(t)&=\frac{1}{h}\left(f-f_h\right)(t).\\
	 \end{split}
	 \end{equation}
	 It follows from the latter definition and \eqref{AffInterpolantDef} that 
	 \begin{equation}\label{AffInterpProp}
	     \begin{split}
	         \tder\widetilde f^N(t)&=\partial^-_{t,h}f^N(t)\,\,\mbox{for all}\,t\in [kh,(k+1)h),\,k\in\mathbb{N}_{0},\\
	         \|\widetilde f^N\|_{L^p(0,\tau;X)}&\leq \|f^N\|_{L^p(0,\tau;X)}+\|f^N_h\|_{L^p(0,\tau;X)}\text{ for any }p\in[1,\infty]\,\,\mbox{and}\,\,\tau>0.
	     \end{split}
	 \end{equation}
	 We will also frequently use the following relation that allows us to express the difference of values at certain times of a piecewise constant intepolant  by the corresponding difference for a piecewise affine interpolant. Let $s=\tilde mh$ for $\tilde m\in\eN_0$ be given. We consider for $t\geq 0$ the difference $f^N(t+s)-f^N(t)$. Obviously, there is $\tilde k\in\eN_0$ such that $t\in [\tilde k h,(\tilde k+1)h)$ and $t+s\in [(\tilde k+\tilde m) h,(\tilde k+\tilde m+1)h)$. Then by the definitions of interpolants we obtain
	 \begin{equation}\label{InterpTimeDiffIdent}
	    f^N(t+s)-f^N(t)=f_{\tilde k+\tilde m+1}-f_{\tilde k+1}=\tilde f^N\left((\tilde k+\tilde m+1)h\right) - \tilde f^N\left((\tilde k+1)h\right)=\widetilde f^N(\tilde t+s)- \widetilde f^N\left(\tilde t\right)
	 \end{equation}
	 for $\tilde t=(\tilde k+1)h$.
	 
	 Now we will specify integral identities that are satisfied by interpolants $v^N, M^N, \phi^N, \mu^N$. For arbitrary $\tau\in (0,\infty)$ there exists a $\overline{k}_{\tau}\in\mathbb{N}_{0}$ such that $\tau\in[\overline k_{\tau}h,(\overline{k}_{\tau}+1)h).$ Further for a function $\psi_1\in L^ 2(0,\infty;V(\Omega))$ we set in \eqref{identity1} $\widetilde\psi_1=\int_{kh}^b\psi_1$, where 
	 $$b=\begin{cases}
	 (k+1)h &k<\overline{k}_{\tau},\\
	 \tau &k=\overline{k}_{\tau}
	 \end{cases}$$ 
	 and summing the resulting expressions over $k\in\{0,\ldots,\overline{k}_{\tau}\}$ we obtain
	 \begin{equation}\label{intidentity1}
	 \begin{split}
	 &\int_{0}^{\tau}\int_{\Omega} \partial^{-}_{t,h}(v^{N})\cdot\psi_{1}+\int_{0}^{\tau}\int_{\Omega}(v^{N}\cdot\nabla)v^{N}\cdot\psi_{1}+\int_{0}^{\tau}\int_{\Omega}\nabla\mu^{N}\phi^{N}_{h}\cdot\psi_{1}\\&-\int_{0}^{\tau}\int_{\Omega}\left(\frac{\xi(\phi^{N}_{h})}{\alpha^{2}}(|M^{N}|^{2}M^{N}-M^{N}_{h})\nabla M^{N}\right)\cdot\psi_{1}\\
	 &+\int_{0}^{\tau}\int_{\Omega}\left(\dvr(\xi(\phi^{N}_{h})\nabla M^{N})\nabla M^{N}\right)\cdot\psi_{1}
	 =-\nu\int_{0}^{\tau}\int_{\Omega}\nabla v^{N}\cdot\nabla\psi_{1},
	 \end{split}
	 \end{equation}
	 for all $\psi_{1}\in L^2(0,\infty;V(\Omega))$ and $0<\tau<\infty.$\\ 
	 Similarly using \eqref{identity2}--\eqref{identity4} and recalling that now the role of $M_{0}$ and $\phi_{0}$ are replaced respectively by $M^{N}_{0}$ and $\phi^{N}_{0}$ we infer
	 \begin{equation}\label{intidentity2}
	 \begin{split}
	 &\int_{0}^{\tau}\int_{\Omega}\partial^{-}_{t,h}M^{N}\cdot\psi_{2}+\int_{0}^{\tau}\int_{\Omega}(v^{N}\cdot\nabla)M^{N}\cdot\psi_{2}\\
	 & =\int_{0}^{\tau}\int_{\Omega}\left(\dvr({\xi(\phi^{N}_{h})}\nabla
	 M^{N})-\frac{\xi(\phi^{N}_{h})}{\alpha^{2}}(|M^{N}|^{2}M^{N}-M^{N}_{h})\right)\cdot\psi_{2},
	 \end{split}
	 \end{equation}
	 for all $\psi_{2}\in L^2(0,\infty;W^{1,2}(\Omega)),$ $0<\tau<\infty$ and 
	 \begin{equation}\label{intidentity3}
	 \begin{split}
	 &\int_{0}^{\tau}\int_{\Omega} \partial^{-}_{t,h}\phi^{N}\cdot\psi_{3}+\int_{0}^{\tau}\int_{\Omega}(v^{N}\cdot\nabla)\phi^{N}_{h}\cdot\psi_{3}=-\int_{0}^{\tau}\int_{\Omega}\nabla\mu^{N}\cdot\nabla\psi_{3}
	 \end{split}
	 \end{equation}
	 and 
	 \begin{equation}\label{intidentity4}
	 \begin{split}
	 &\int_{0}^{\tau}\int_{\Omega}\mu^{N}\psi_{3}-\int_{0}^{\tau}\int_{\Omega} H_{0}(\phi^{N},\phi^{N}_{h})\frac{|\nabla M^{N}|^{2}}{2}\psi_{3}-\int_{\Omega}\frac{H_{0}(\phi^{N},\phi^{N}_{h})}{4\alpha^{2}}(|M^{N}|^{2}-1)^{2}\psi_{3}\\
	 &=\eta\int_{0}^{\tau}\int_{\Omega}\nabla\phi^{N}\cdot\nabla\psi_{3}+\frac{1}{\eta}\int_{0}^{\tau}\int_{\Omega}((\phi^{N})^{3}-\phi^{N}_{h})\psi_{3},
	 \end{split}
	 \end{equation}
	 for all $\psi_{3}\in L^2(0,\infty;W^{1,2}(\Omega))$ and $0<\tau<\infty.$\\ 
	\subsection{Compactness of a sequence of interpolants}
	In this section we will provide arguments to pass to the limit $h\longrightarrow 0$ (equivalently $N\longrightarrow\infty$) in \eqref{intidentity1}--\eqref{intidentity4} in order to prove the existence of a weak solution to problem \eqref{diffviscoelastic*}.
	\subsubsection{Obtaining convergences in a weak sense}
	  All the convergences necessary for that passage are consequences of uniform bounds following from the energy inequality for interpolants $v^N, M^N, \phi^N, \mu^N$, which we now derive. Summing \eqref{discreteestimate} over $k$ we conclude
	 \begin{equation}\label{energyinterpolant}
	 \begin{split}
	 & E_{tot}(v^{N}(t),M^{N}(t),\phi^{N}(t))\\
	 & +\int_{0}^{t}\int_{\Omega}\left(\nu|\nabla v^{N}|^{2}+|\nabla\mu^{N}|^{2}+
	 \left|\mathrm{div}(\xi(\phi^{N}_{h})\nabla M^{N})-\frac{\xi(\phi^{N}_{h})}{\alpha^{2}}(|M^{N}|^{2}M^{N}-M^{N}_{h})\right|^{2}\right)\\
	 &\leqslant E_{tot}(v_0,M^{N}_0,\phi^{N}_0)
	 \end{split}
	 \end{equation}
	 for each $t\in h\mathbb{N}_{0}$. Moreover, taking into account that all the quantities involved in \eqref{energyinterpolant} are constant on intervals of the form $[kh,(k+1)h)$, $k\in\eN_{0}$, we conclude that \eqref{energyinterpolant} is satisfied for all $0< t< \infty$.
	 At this moment one recalls the definition of $E_{tot}$ from \eqref{defEtot}. The boundedness of $E_{tot}(v_{0},M^{N}_{0},\phi^{N}_{0})$ that follows from \eqref{MN0Cnv} and \eqref{PhiN0Cnv} implies:
	 \begin{alignat}{2}
	    &\{v^{N}\}&&\text{ is bounded in }L^{2}(0,T+1;W^{1,2}(\Omega))\cap L^{\infty}(0,T+1;L^{2}(\Omega)),\label{VNBound}\\
	    &\{M^{N}\}&&\text{ is bounded in } L^{\infty}(0,T+1;W^{1,2}(\Omega)),\label{MNBound}\\
	    & \{\phi^{N}\}&&\text{ is bounded in } L^{\infty}(0,T+1;W^{1,2}(\Omega)),\label{PhiNBound}\\
	    &\{\nabla\mu^{N}\} &&\text{ is bounded in } L^{2}(0,T+1;L^{2}(\Omega)),\label{GMUNBound}
	 \end{alignat}
	 \begin{equation}\label{MagnDissNBound}
	      \left\{\mathrm{div}(\xi(\phi^{N}_{h})\nabla M^{N})-\frac{\xi(\phi^{N}_{h})}{\alpha^{2}}(|M^{N}|^{2}M^{N}-M^{N}_{h})\right\}\text{ is bounded in }L^{2}(0,T+1;L^{2}(\Omega)).
	 \end{equation}
	 We note that bound \eqref{MNBound} follows from the uniform bound of $\{\nabla M^N\}$ in $L^2(0,T+1;L^2(\Omega))$ and the bound of $\{M^N\}$ in $L^\infty(0,T+1;L^4(\Omega))$, which is a direct consequence of \eqref{XiAssum}$_2$ and the uniform bound of $\{\xi(\phi^N)(|M^N|^2-1)^2\}$ in $L^\infty(0,T+1;L^1(\Omega))$.
	    By Lemma~\ref{boundintmu} we get
	 \begin{equation*}	     
	    \int_0^{T+1}\left|\int_\Omega \mu^N\right|\leqslant G(T+1)
	 \end{equation*}
	 for a monotone function $G:\mathbb{R}^{+}\to\mathbb{R}^{+}$. Combining the above estimate and \eqref{GMUNBound} we infer using Lemma~\ref{Lem:PoincareAverage}
	  \begin{equation}\label{MuNBound}	     
	    \{\mu^N\}\text{ is bounded in }L^2(0,T+1;W^{1,2}(\Omega))
	 \end{equation}
	 By definition, $\phi^{N}_{h}=\phi^{N}(t-h)$ and it coincides with $\phi^{N}_{0}$ in $[-h,0),$ this provides 
	 \begin{equation}\label{phiNh}
	 \phi^{N}_{h}\,\,\mbox{is bounded in}\,\, L^{\infty}(0,T+1;W^{1,2}(\Omega)).
	 \end{equation}
	 We obtain similarly 
	 \begin{equation}\label{MNh}
	  M^{N}_{h}\,\,\mbox{is bounded in}\,\, L^{\infty}(0,T+1;W^{1,2}(\Omega))
	  \end{equation}
	  and 
	  \begin{equation}\label{VNh}
	  v^{N}_{h}\,\,\mbox{is bounded in}\,\, L^{\infty}(0,T+1;L^2(\Omega)).
	  \end{equation}
	  In view of \eqref{VNBound}--\eqref{PhiNBound} and \eqref{MuNBound}, one has the following weak type convergences upto some subsequence (not explicitly relabeled):
	  \begin{equation}\label{weakconvergences}
	  \begin{alignedat}{2}
	   v^{N}&\rightharpoonup v&&\mbox{ in }L^{2}(0,T;W^{1,2}(\Omega)),\\[3.mm]
	    v^{N}&\rightharpoonup^* v&&\mbox{ in } L^{\infty}(0,T;L^{2}(\Omega)),\\[3.mm]
	    M^{N}&\rightharpoonup^* M &&\mbox{ in } L^{\infty}(0,T;W^{1,2}(\Omega)),\\[3.mm]
	   \phi^{N}&\rightharpoonup^* \phi &&\mbox{ in } L^{\infty}(0,T;W^{1,2}(\Omega)),\\[3.mm]
	   \mu^{N}&\rightharpoonup \mu &&\mbox{ in }L^{2}(0,T;W^{1,2}(\Omega)).\\[3.mm]
	  \end{alignedat}
	  \end{equation}
	 \subsubsection{Recovering strong convergences and related results}
	  \begin{itemize}
	  	\item \textit{Relative compactness of }$\mathit{\{v^N\}}$ \textit{w.r.t. the strong topology of} $\mathit{L^2(0,T;L^4(\Omega))}$\textit{ and relative compactness of }$\{\widetilde{v}^{N}\}$ \textit{w.r.t. the weak}$^*$\textit{ topology of}  $\mathit{L^{\infty}(0,T;L^{2}(\Omega))}$:\\
	  First we claim that up to a nonrelabeled subsequence 
     	  	\begin{equation*}
	  	    v^N\to v\text{ in }L^2(0,T;L^4(\Omega))\text{ as }N\to\infty.
	  	\end{equation*}
	  In that direction we first show that $\partial_{t,h}^{-}v^{N}$ is bounded in $L^{2}(0,T+1;(V(\Omega))'),$ where $(V(\Omega))'$ is the dual of $V(\Omega)$ with $\LND$ as the pivot space. To prove our claim we recall the identity \eqref{intidentity1} with $\psi_{1}\in L^{2}(0,\infty;V(\Omega))\hookrightarrow L^{2}(0,\infty;L^{\infty}(\Omega)).$ We will frequently use the embedding $W^{1,2}(\Omega)\hookrightarrow L^{6}(\Omega)$ and $L^{6}(\Omega)\times L^{2}(\Omega)\hookrightarrow L^{\frac{3}{2}}(\Omega)\hookrightarrow L^{1}(\Omega)$ in the following computation. One obtains from \eqref{intidentity1} that
	  	\begin{equation*}
	  	\begin{split}
	  	&\left|\int_0^{T+1} \langle \partial^{-} _{t,h}{v^N},\psi^1\rangle\right|\leqslant C\left(\|v^{N}\|_{L^{\infty}(0,T+1;L^{2}(\Omega))}\|\nabla v^{N}\|_{L^{2}(0,T+1;L^{2}(\Omega))}\|\psi_{1}\|_{L^{2}(0,T+1;L^{\infty}(\Omega))}\right.\\[3.mm]
	  	& +\|\nabla\mu^{N}\|_{L^{2}(0,T+1;L^{2}(\Omega))}\|\phi^{N}_{h}\|_{L^{\infty}(0,T+1;L^{6}(\Omega))}\|\psi_{1}\|_{L^{2}(0,T+1;L^{\infty}(\Omega))}\\[3.mm]
	  	& +\|\mathrm{div}(\xi(\phi^{N}_{h})\nabla M^{N})-\frac{\xi(\phi^{N}_{h})}{\alpha^{2}}(|M^{N}|^{2}M^{N}-M^{N}_{h})\|_{L^{2}(0,T+1;L^{2}(\Omega))}\|\nabla M^{N}\|_{L^{\infty}(0,T+1;L^{2}(\Omega))}\\[3.mm]
	  	&\quad\|\psi_{1}\|_{L^{2}(0,T+1;L^{\infty}(\Omega))}\\[3.mm]
	  	& +\|\nabla v^{N}\|_{L^{2}(0,T+1;L^{2}(\Omega))}\|\nabla\psi_{1}\|_{L^{2}(0,T+1;L^{6}(\Omega))}.
	  	\end{split}
	  	\end{equation*}
	  	Consequently, one has that 
	  	\begin{equation}\label{VNDiscTimeDerEst}
	  	    \{\partial^{-}_{t,h}v^{N}\}\text{ is bounded in } (L^{2}(0,T+1;V(\Omega)))'=L^{2}(0,T+1;(V(\Omega))').
	  	\end{equation}
	  	Next, combining \eqref{VNDiscTimeDerEst} with \eqref{AffInterpProp}$_1$, \eqref{VNBound} and \eqref{VNh} with \eqref{AffInterpProp}$_2$ we deduce that
	  	\begin{equation}\label{TildeVNBound}
	  	\{\widetilde v^N\}\text{ is bounded in }W^{1,2}(0,T+1;(V(\Omega))').
	  	\end{equation}
	  	Using \eqref{InterpTimeDiffIdent}, the latter bound and the embedding $W^{1,2}(0,T+1;(V(\Omega))')\hookrightarrow C^{0,\frac{1}{2}}([0,T+1];(V(\Omega))')$ we obtain 
	  	\begin{equation*} 
	  	    \|v^N(t+\tilde s)-v^N(t)\|_{(V(\Omega))'}=\|\widetilde v^N(\tilde t+\tilde s)-\widetilde v^N(\tilde t)\|_{(V(\Omega))'} \leq c\tilde s^\frac{1}{2}
	  	\end{equation*}
	  	for $t\in [0,T+1-\tilde s]$ with $\tilde s=\tilde mh$, $\tilde m\in\eN$ and $\tilde{s}<T+1.$ Hence we conclude
	  	\begin{equation*} 
	  	    \int_0^{T+1-\tilde s}\|v^N(t+\tilde s)-v^N(t)\|^2_{(V(\Omega))'}\leq c(T+1)\tilde s
	  	\end{equation*}	  
	  	with $c$ independent of $N$. Then we find $m\in\eN$ such that $T<mh\leq T+1$. As a consequence of Lemma~\ref{Lem:RefIneq} we have \begin{equation*} 
	  	    \int_0^{T-s}\|v^N(t+s)-v^N(t)\|^2_{(V(\Omega))'}\leq c(T+1)s
	  	\end{equation*} 
	  	for any $0<s<T.$ Taking also into account \eqref{VNBound} and the chain of embeddings $\WND\stackrel{C}{\hookrightarrow} L^4(\Omega)\hookrightarrow (V(\Omega)')$ Lemma~\ref{Lem:RelComp}  yields the existence of a nonrelabeled subsequence $\{v^N\}$ such that 
	  	\begin{equation}\label{VNStrongly}
	  	    v^N\to v\text{ in }L^2(0,T;L^4(\Omega))\text{ as }N\to\infty,
	  	\end{equation}
	  	and hence our claim.\\
	  	Our second claim is that up to a nonrelabeled subsequence
	 	\begin{equation}\label{tildVNWeaklyS}
	 	\widetilde v^N\rightharpoonup^* v\text{ in }L^\infty(0,T;L^2(\Omega))\text{ as }N\to\infty.
	  	\end{equation}
	 One observes
	 	\begin{equation}\label{stconv2}
	 	\widetilde{v}^{N}(t)-v^{N}(t)=(t-(k+1)h)\partial_{t,h}^{-}v^{N}(t),\quad\mbox{for}\quad t\in[kh,(k+1)h),\quad k\in\mathbb{N}_{0}.
	  	\end{equation}
	  	Since $|t-(k+1)h|\leqslant h=\frac{1}{N}\leqslant 1,$ \eqref{AffInterpProp}$_1$ and \eqref{stconv2} lead to
	  	\begin{equation*}
	   \|\widetilde{v}^{N}(t)-v^{N}(t)\|_{(V(\Omega))'}\leqslant h\|\partial_{t}\widetilde{v}^{N}(t)\|_{(V(\Omega))'}\quad\mbox{for}\quad t\in[0,\infty).
	 	\end{equation*}
	 	Hence \eqref{TildeVNBound} implies 
	  	\begin{equation}\label{TildeVNVNDiff}
	   \widetilde{v}^{N}-v^{N}\to 0\text{ in }L^2(0,T;(V(\Omega))')\text{ as }N\to\infty,
	 	\end{equation}
	 	from which 
	 	\begin{equation}\label{TVNStrongly}
	 	    \widetilde{v}^{N}\to v\text{ in }L^2(0,T;(V(\Omega))')\text{ as }N\to \infty
	 	\end{equation}    
	 	immediately follows by \eqref{VNStrongly} as 
	 	\begin{equation*}
	 	    \|\widetilde v^N-v\|_{L^2(0,T;(V(\Omega))')}\leq \|\widetilde v^N-v^N\|_{L^2(0,T;(V(\Omega))')}+\| v^N-v\|_{L^2(0,T;(V(\Omega))')}.
	 	\end{equation*}
	 	Hence in view of the bound of $\{\widetilde {v}^{N}\}$ in $L^{\infty}(0,T;L^{2}(\Omega))$ (which follows immediately by using \eqref{VNBound}, \eqref{VNh} and \eqref{AffInterpProp}$_{2}$) we conclude \eqref{tildVNWeaklyS}.
	\item \textit{Relative compactness of }$\mathit{\{\phi^N\},\{\phi_h^N\}}$\textit{ and }$\mathit{\{\widetilde \phi^N\}}$\textit{ w.r.t. the strong topology of } $\mathit{L^2(0,T;L^4(\Omega))}$:\\
	In that direction we will first verify that $\partial^{-}_{t,h}\phi^{N}$ is bounded in $L^{2}(0,T+1;(W^{1,2}(\Omega))').$ One uses \eqref{VNBound} and \eqref{phiNh} to furnish that $(v^{N}\cdot\nabla)\phi^{N}_{h}$ is bounded in $L^{2}(0,T+1;L^{\frac{3}{2}}(\Omega)),$ since
	  	\begin{equation*}
	  	\begin{split}
	  	L^{2}(0,T+1;W^{1,2}(\Omega))\times L^{\infty}(0,T+1;L^{2}(\Omega))&\hookrightarrow L^{2}(0,T+1;L^{6}(\Omega))\times L^{\infty}(0,T+1;L^{2}(\Omega))\\
	  	&\hookrightarrow L^{2}(0,T+1;L^{\frac{3}{2}}(\Omega)).
	  	\end{split}
	  	\end{equation*}
	  	Further one recalls from \eqref{MuNBound} that $\nabla\mu^{N}$ is bounded in $L^{2}(0,T+1;L^{2}(\Omega))$. Hence in view of \eqref{intidentity3} we conclude that
	  	\begin{equation}\label{TDerPhiNEst}
	   \{\partial^{-}_{t,h}\phi^{N}\}\text{ is bounded in }L^{2}(0,T+1;(W^{1,2}(\Omega))').
	    \end{equation}
	    Proceeding in the same way as we did in the previous item ($i.e$ as we have shown the bound of $\{\widetilde v^{N}\}$ in $W^{1,2}((W^{1,2}(\Omega))')\cap L^{\infty}(L^{2}(\Omega))$ and the convergence \eqref{VNStrongly} of $v^{N}$), we combine the properties of interpolants from \eqref{AffInterpProp} with bounds \eqref{GMUNBound} and \eqref{phiNh} to deduce that 
	    \begin{equation}\label{WTildeBound}
	        \{\widetilde \phi^N\}\text{ is bounded in }W^{1,2}(0,T+1;(W^{1,2}(\Omega))')\cap L^\infty(0,T+1;W^{1,2}(\Omega))
	    \end{equation}
	    and further due to the chain of the embeddings $W^{1,2}(\Omega)\stackrel{C}{\hookrightarrow}L^4(\Omega)\hookrightarrow(W^{1,2}(\Omega))'$ deduce the existence of a nonrelabeled subsequence $\{\phi^N\}$ such that 
	    \begin{equation}\label{PhiNStrongly}
	        \phi^N\to \phi\text{ in }L^2(0,T;L^4(\Omega))\text{ as }N\to\infty.
	    \end{equation}
	    The bounds in \eqref{WTildeBound} imply by the Aubin-Lions lemma the relative compactness of $\{\widetilde \phi^N\}$ w.r.t. the strong topology of $L^2(0,T;L^4(\Omega))$. Imitating steps leading to \eqref{TildeVNVNDiff} and \eqref{TVNStrongly} we obtain that up to a nonrelabeled subsequence
	    \begin{equation*}
	        \widetilde \phi^N\to \phi\text{ in }L^2(0,T;(W^{1,2}(\Omega))')\text{ as }N\to\infty.
	    \end{equation*}
	    Consequently, it follows that up to a nonrelabeled subsequence
	    \begin{equation}\label{TildePhiNStrongly}
	        \widetilde \phi^N\to \phi\text{ in }L^2(0,T;L^4(\Omega))\text{ as }N\to\infty.
	    \end{equation}
	    It remains to show that up to a nonrelabeled subseqence
	    \begin{equation}\label{PhiNhStrongly}
	        \phi^N_h\to \phi\text{ in }L^2(0,T; L^4(\Omega))\text{ as }N\to\infty.
	    \end{equation}
	    We recall that the Sobolev embedding implies $W^{\frac{3}{4},2}(\Omega)\hookrightarrow L^{4}(\Omega)$ and by the interpolation inequality, cf.\ \cite[Theorem 12.5.]{LioMag}, we get
	    \begin{equation}\label{PhiNNhDif}
	    \begin{split}
	       \|\phi^N_h-\phi^N\|_{L^2(0,T;L^4(\Omega))}\leq & c\|\phi^N_h-\phi^N\|_{L^2(0,T;W^{\frac{3}{4},2}(\Omega))}\\
	       \leq &c\|\phi^N_h-\phi^N\|^\frac{7}{8}_{L^2(0,T;W^{1,2}(\Omega))}\|\phi^N_h-\phi^N\|^\frac{1}{8}_{L^2(0,T;(W^{1,2}(\Omega))^{'})}.
	       \end{split}
	    \end{equation}
 In view of \eqref{shifttimef}, \eqref{AffInterpProp}$_{1}$ and the uniform bound \eqref{WTildeBound} we infer
	\begin{equation*}
	    \phi^N_h-\phi^N\to 0\text{ in }L^2(0,T;(W^{1,2}(\Omega))^{'})
	\end{equation*}
	as $N\to\infty$ (or equivalently as $h\rightarrow 0$). Moreover, recalling \eqref{GMUNBound} and \eqref{phiNh}, we deduce from \eqref{PhiNNhDif} that 
	\begin{equation}\label{PhiNhNDiffConv}
	       \phi^N_h-\phi^N\to 0\text{ in }L^2(0,T;L^4(\Omega)).
	        \end{equation}
	        As 
	        \begin{equation*}
	            \|\phi^N_h-\phi\|_{L^2(0,T;L^4(\Omega))}\leq \|\phi^N_h-\phi^N\|_{L^2(0,T;L^4(\Omega))}+\|\phi^N-\phi\|_{L^2(0,T;L^4(\Omega))},
	        \end{equation*}
	        we conclude \eqref{PhiNhStrongly} by \eqref{PhiNhNDiffConv} and \eqref{PhiNStrongly}.
	       	\item \textit{Relative compactness of }$\mathit{\{M^N\}}$\textit{ w.r.t. the strong topology of $L^{8}(0,T;L^{4}(\Omega))\cap L^{2}(0,T;W^{1,2}(\Omega)),$ } $\mathit{\{\widetilde{M}^N\}}$\textit{ w.r.t. the strong topology of $L^{2}(0,T;L^{4}(\Omega))$ } and $\mathit{\{M^N_{h}\}}$\textit{ w.r.t. the strong topology of $L^{2}(Q_{T})$ } : \\
	  	In that direction we first show that $\{\partial^{-}_{t,h}M^{N}\}$ is bounded in $L^{2}(0,T+1;L^{\frac{3}{2}}(\Omega))$. One first uses \eqref{VNBound} and \eqref{MNBound} to furnish that $\{(v^{N}\cdot\nabla)M^{N}\}$ is bounded in $L^{2}(0,T+1;L^{\frac{3}{2}}(\Omega)).$ Then in view of \eqref{PhiNBound} one has that\\ $\left\{\mathrm{div}(\xi(\phi^{N}_{h})\nabla M^{N})
	  	-\frac{\xi(\phi^{N}_{h})}{\alpha^{2}}(|M^{N}|^{2}M^{N}-M^{N}_{h})\right\}$ is bounded in $L^{2}(0,T+1;L^{2}(\Omega)).$ Hence using test functions $\psi_{2}\in L^{2}(0,T+1;L^{3}(\Omega))$ in \eqref{intidentity2}, which is legal due to the density of $W^{1,2}(\Omega)$ in $L^3(\Omega)$, one at once has that
	  	\begin{equation}\label{TDerMNEst}
	  	    \{\partial^{-}_{t,h}M^{N}\}\text{ is bounded in } L^{2}(0,T+1;L^{\frac{3}{2}}(\Omega)).
	  	\end{equation}
	  	Uniform bound \eqref{TDerMNEst} along with \eqref{AffInterpProp}, bounds \eqref{MNBound}, \eqref{MNh} and the definition of $\widetilde{M}^{N}$ (cf.\ \eqref{AffInterpolantDef}) furnish
	  	\begin{equation}\label{TildeMNBound}
	  	\{\widetilde M^N\}\text{ is bounded in }W^{1,2}(0,T+1;L^\frac{3}{2}(\Omega))\cap L^{\infty}(0,T+1;W^{1,2}(\Omega)).
	  	\end{equation}
	  	Since $W^{1,2}(\Omega)\stackrel{C}{\hookrightarrow}L^4(\Omega)\hookrightarrow L^\frac{3}{2}(\Omega),$ one can now imitate arguments leading to \eqref{VNStrongly} to conclude that up to a nonrelabeled subsequence
	  	\begin{equation*}
	  	    M^N\to M\text{ in }L^2(0,T;L^4(\Omega))\text{ as }N\to\infty.
	  	\end{equation*}
	  	We improve the latter strong convergence. For that one observes the following by interpolation:
	  	\begin{equation}\label{improvingstrngconv}
	  	\|M^{N}-M\|_{L^{8}(0,T;L^{4}(\Omega))}\leqslant \|M^{N}-M\|^{\frac{3}{4}}_{L^{\infty}(0,T;L^{6}(\Omega))}\|M^{N}-M\|^{\frac{1}{4}}_{L^{2}(0,T;L^{2}(\Omega))}.
	  	\end{equation}
	  	The boundedness of the first multiplier appearing on the right hand side of \eqref{improvingstrngconv} and the fact that the second multiplier converges to zero at once render:
	  	\begin{equation}\label{strngconvMnl8l4}
	  	M^{N}\longrightarrow M\text{ in } L^{8}(0,T;L^{4}(\Omega))\text{ as }N\to\infty.
	  	\end{equation}
	  	Similar arguments which were used to show \eqref{PhiNhStrongly}, namely the interpolation
	  	\begin{equation*}
	  	    \|M^N_h-M^N\|_{L^2(0,T;L^2(\Omega))}\leq c\|M^N_h-M^N\|^{\frac{1}{2}}_{L^2(0,T;W^{1,2}(\Omega))}\|M^N_h-M^N\|^\frac{1}{2}_{L^2(0,T;(W^{1,2}(\Omega))^{'})}
	  	\end{equation*}
	  	can also be employed to show that
	  	\begin{equation}\label{strngconvMnh}
	  	M^{N}_{h}\longrightarrow M\text{ in } L^{2}(Q_{T})\text{ as }N\to\infty,
	  	\end{equation}
	  	of course this can be improved but \eqref{strngconvMnh} is enough for our purpose.\\
	  	Next we want to improve the compactness of $\nabla M^{N}$ w.r.t. the strong topology of $L^{2}(0,T;L^{2}(\Omega))$ or in other words we will show that
	  	\begin{equation}\label{MNStrongC}
	  	M^{N}\longrightarrow M\quad\mbox{in}\quad L^{2}(0,T;W^{1,2}(\Omega)).
	  	\end{equation}
	  	For the proof of \eqref{MNStrongC} we will exploit the monotonic structure of $\dvr(\xi(\phi^{N}_{h})\nabla M^{N}).$\\
	  	 In the direction of the proof of \eqref{MNStrongC} one first recalls \eqref{TildeMNBound} and applies the Aubin-Lions lemma to obtain the relative compactness of $\{\widetilde M^N\}$ w.r.t. the strong topology of $L^2(0,T;L^4(\Omega))$. Since up to a nonrelabeled subsequences
	  	$\|\widetilde M^N-M^N\|_{L^2(0,T;L^\frac{3}{2}(\Omega))}\to 0,$ which can be obtained using the uniform bound of $\partial_{t}\widetilde{M}^{N}$ in $L^{2}(0,T;L^{\frac{3}{2}}(\Omega))$ (cf.\ \eqref{TildeMNBound}) and similar line of arguments used to show \eqref{TildeVNVNDiff}, we conclude
	  	\begin{equation}\label{strngconvtM}
	  	\widetilde{M}^{N} \to {M} \text{ in } L^{2}(0,T;L^{4}(\Omega))\text{  as }N\to\infty.
	  	\end{equation}
	  	Further it is worth noticing that due to \eqref{TildeMNBound} for a not explicitly relabeled subsequence of  $\{\partial_{t}\widetilde{M}^{N}\}$ the following holds
	  	\begin{equation}\label{weakconvdttMn}
	  	\partial_{t}\widetilde{M}^{N}\rightharpoonup \partial_{t}M\text{ in } L^{2}(0,T;L^{\frac{3}{2}}(\Omega))\text{ as }N\to\infty.
	  	\end{equation} 
	  	Next fixing $\tau=T$, setting $\psi_2(t)=0$ for $t>T$, using \eqref{AffInterpProp}$_{1}$ and integrating by parts in the first term on the right hand side of \eqref{intidentity2} we get 
	  	 \begin{equation}\label{intidentity*}
	  	 \begin{split}
	  	 &\int_{0}^{T}\int_{\Omega}\partial_{t}\widetilde{M}^{N}\cdot\psi_{2}+\int_{0}^{T}\int_{\Omega}(v^{N}\cdot\nabla)M^{N}\cdot\psi_{2}\\
	  	 & =-\int_{0}^{T}\int_{\Omega}\xi(\phi^{N}_{h})\nabla
	  	 M^{N}\cdot\nabla \psi_{2}-\int_{0}^{T}\int_{\Omega}{\xi(\phi^{N}_{h})}(|M^{N}|^{2}M^{N}-M^{N}_{h})\cdot\psi_{2},
	  	 \end{split}
	  	 \end{equation}
	  	where we have used that for a.e $0<t<T,$ $\partial_{n}M^{N}=0$ on $\partial\Omega,$ following its construction.\\
	  	Our goal now is to perform the passage $N\to\infty$ in the latter identity. To this end we need to justify convergences of terms appearing on the right hand side of \eqref{intidentity*}.	Using \eqref{PhiNhStrongly} one has the a.e.\ convergence of $\phi^{N}_{h}$ to $\phi.$ One then uses \eqref{XiAssum} to have up to a nonrelabeled subsequence
	   \begin{equation}\label{convxiphinh}
	   \xi(\phi^{N}_{h})\longrightarrow \xi(\phi)\quad\mbox{ a.e.\ in }Q_T
	   \end{equation} 
	   and by the Lebesgue dominated convergence theorem
	   also
	    \begin{equation}\label{convxiphinhstrong}
	    \xi(\phi^{N}_{h})\longrightarrow \xi(\phi)\quad\mbox{ in }L^q(Q_T),\text{ for any } 1\leq q<\infty.
	   \end{equation} 
	   The strong convergence \eqref{strngconvMnl8l4} and the weak$^{*}$ convergence of $M^{N}$ to $M$ in $L^{\infty}(0,T;L^{6}(\Omega))$ (which follows from \eqref{weakconvergences}$_3$ are enough to conclude that
	   \begin{equation}\label{weakconvMn2M}
	   \begin{array}{l}
	   |M^{N}|^{2}M^{N}\rightharpoonup |M|^{2}M\quad\mbox{in}\quad L^{4}(0,T;L^{\frac{3}{2}}(\Omega)).
	   \end{array}
	   \end{equation}
	   Next we observe that due to \eqref{XiAssum}$_2$ and \eqref{MNBound} we get that $\{\xi(\phi_h^N)\nabla M^N\}$ is bounded uniformly in $L^2(Q_T)$ therefore the latter sequence possesses a not explicitly relabeled subsequence that converges weakly in $L^2(Q_T)$. We identify the weak limit by combining \eqref{convxiphinhstrong} with \eqref{weakconvergences}$_3$ and deduce that
	   \begin{equation}\label{NonlL2weakly}
	   \xi(\phi_h^N)\nabla M^N\rightharpoonup \xi(\phi)\nabla M\text{ in }L^2(Q_T).
	   \end{equation}
	   Let us now perform the passage $N\to\infty$ in \eqref{intidentity*}. Since $\psi_{2}\in L^{2}(0,T;W^{1,2}(\Omega)),$ the passage to the limit $N\to \infty$ in both terms on the left hand side of \eqref{intidentity*} is straightforward  as \eqref{weakconvdttMn}, \eqref{VNStrongly} and \eqref{weakconvergences}$_3$ are available. For the passage in the terms on the right hand side we apply \eqref{NonlL2weakly} and \eqref{weakconvMn2M} together with \eqref{convxiphinhstrong}, \eqref{strngconvMnh} respectively. Consequently one has
	   \begin{equation}\label{intidentitylimit*}
	   \begin{split}
       &\int_{0}^{T}\int_{\Omega}\partial_{t}{M}\cdot\psi_{2}+\int_{0}^{T}\int_{\Omega}(v\cdot\nabla)M\cdot\psi_{2}\\
	   & =-\int_{0}^{T}\int_{\Omega}\xi(\phi)\nabla
	   M\cdot\nabla \psi_{2}-\int_{0}^{T}\int_{\Omega}{\xi(\phi)}(|M|^{2}M-M)\cdot\psi_{2}.
	   \end{split}
	   \end{equation}
	   We consider the difference of \eqref{intidentity*} and \eqref{intidentitylimit*} and further set $\psi_{2}=(M^{N}-M),$ which is possible since $(M^{N}-M)\in L^{2}(0,T;W^{1,2}(\Omega))$:
	   \begin{equation}\label{intidentitydifference*}
	   \begin{split}
	   &\int_{0}^{T}\int_{\Omega}\partial_{t}(\widetilde{M}^{N}-M)\cdot(M^{N}-M)+\int_{0}^{T}\int_{\Omega}((v^{N}\cdot\nabla)M^{N}-(v\cdot\nabla)M)\cdot(M^{N}-M)\\
	   & =\int_{0}^{T}\int_{\Omega}(\xi(\phi)\nabla
	   M-\xi(\phi^{N}_{h})\nabla
	   M^{N})\cdot\nabla (M^{N}-M)\\
	   & -\int_{0}^{T}\int_{\Omega}({\xi(\phi^{N}_{h})}(|M^{N}|^{2}M^{N}-M^{N}_{h})-{\xi(\phi)}(|M|^{2}M-M))\cdot(M^{N}-M).
	   \end{split}
	   \end{equation}
	   The identity \eqref{intidentitydifference*} can be re-written in the form:
	   \begin{equation}\label{intidentitydifferencerewritten}
	   \begin{split}
	   & \int_{0}^{T}\int_{\Omega} \xi(\phi^{N}_{h})|\nabla(M^{N}-M)|^{2}\\
	   &=-\int_{0}^{T}\int_{\Omega}\left(\xi(\phi^{N}_{h})-\xi(\phi)\right)\nabla M\cdot\nabla(M^{N}-M)
	    +\int_{0}^{T}\int_{\Omega}\partial_{t}(\widetilde{M}^{N}-M)\cdot(M^{N}-M)\\
	   &\quad +\int_{0}^{T}\int_{\Omega}\left((v^{N}\cdot\nabla)M^{N}-(v\cdot\nabla)M\right)\cdot(M^{N}-M)\\
	   &\quad +\int_{0}^{T}\int_{\Omega}\left({\xi(\phi^{N}_{h})}(|M^{N}|^{2}M^{N}-M^{N}_{h})-{\xi(\phi)}(|M|^{2}M-M)\right)\cdot(M^{N}-M)=\sum_{i=1}^{4}I^N_i.
	   \end{split}
	    \end{equation}
	    Our goal is to show that all $I^N_i$ vanish in the limit $N\to\infty$. In order to handle $I^N_1$, one first uses \eqref{convxiphinh}, \eqref{XiAssum}$_{2}$ and the Lebesgue dominated convergence theorem to show that
	    $$\left(\xi(\phi^{N}_{h})-\xi(\phi)\right)\nabla M\longrightarrow 0\quad\mbox{in}\quad L^{2}(Q_{T}).$$
	     This convergence along with \eqref{weakconvergences}$_3$ proves that $I^{N}_{1}$ converges to zero as $N\rightarrow\infty$. In view of \eqref{strngconvMnl8l4} and \eqref{weakconvdttMn} $I^N_2$ vanishes in the limit $N\to\infty$. $I^N_3$ tends to zero because of \eqref{strngconvMnl8l4}, \eqref{VNStrongly} and \eqref{weakconvergences}$_3$. Finally we will show that $I^{N}_{4}$ converges to zero as $N\rightarrow\infty.$ In view of \eqref{weakconvMn2M}, \eqref{strngconvMnh} and \eqref{convxiphinhstrong} with $q=12$ one in particular observes that
	     \begin{equation}\label{weakconvpolynom}
	     \begin{array}{l}
	     \left({\xi(\phi^{N}_{h})}(|M^{N}|^{2}M^{N}-M^{N}_{h})-{\xi(\phi)}(|M|^{2}M-M)\right)\rightharpoonup 0\quad\mbox{in}\quad L^{\frac{12}{7}}(0,T;L^{\frac{4}{3}}(\Omega)),
	     \end{array}
	     \end{equation}
	     which along with the strong convergence \eqref{strngconvMnl8l4} concludes that $I^N_4$ vanishes in the limit $N\to\infty$ as well.
	     Passing to the limit $N\to\infty$ on the bothe sides of \eqref{intidentitydifferencerewritten} we deduce due to \eqref{XiAssum}$_{2}$
	     \begin{equation*}
	         \lim_{N\to\infty}\|\nabla (M^N-M)\|_{L^2(Q_T)}^2\leq \lim_{N\to\infty}c_1^{-1} \int_{0}^{T}\int_{\Omega} \xi(\phi^{N}_{h})|\nabla(M^{N}-M)|^{2}= 0.
	     \end{equation*} This along with \eqref{strngconvMnl8l4} implies the desired claim \eqref{MNStrongC}.
	  \end{itemize}

	   \subsubsection{Identifying the weak limit of $\left\{\dvr\left(\xi(\phi^N_h)\nabla M^N\right)-\frac{\xi(\phi^N_h)}{\alpha^2}\left(|M^N|^2M^N-M^N_h\right)\right\}$ in $L^2(Q_T)$}
	  The goal now is to show that 
	  \begin{equation}\label{WeakCnvMonster}
	     \dvr\left(\xi(\phi^N_h)\nabla M^N\right)-\frac{\xi(\phi^N_h)}{\alpha^2}(|M^N|^2M^N-M^N_h)\rightharpoonup \dvr\left(\xi(\phi)\nabla M\right)-\frac{\xi(\phi)}{\alpha^2}(|M|^2M-M)\text{ in }L^2(Q_T)
	  \end{equation}
	  for a nonrelabeled subsequence, where the functions $\phi$ and $M$ were obtained in \eqref{weakconvergences}. We note that due to the regularity of $M$ the divergence of $\xi(\phi)\nabla M$ is understood in distributional sense.
	  Bound \eqref{MagnDissNBound} implies that the sequence under consideration possesses a weakly convergent subsequence in $L^2(Q_T)$, which we do not explicitly relabel. The remaining task is to identify the limit. To this end one observes from \eqref{NonlL2weakly} that
	  $$\int_{0}^{T}\int_{\Omega}\xi(\phi^N_h)\nabla M^N\nabla\psi\longrightarrow \int_{0}^{T}\int_{\Omega}\xi(\phi)\nabla M\nabla\psi,\quad\text{ for all }\,\psi\in\mathcal{D}(Q_T), $$
	  $i.e$ the convergence 
	  $$\dvr\left(\xi(\phi^N_h)\nabla M^N\right)\longrightarrow \dvr\left(\xi(\phi)\nabla M\right)$$
	  holds in $\mathcal{D}'(Q_T)$ ($i.e.$ in the sense of distribution).\\
	  In view of this distributional convergence along with \eqref{weakconvpolynom} one identifies the weak limit and concludes the proof of \eqref{WeakCnvMonster}.
	  \subsection{The energy inequality for the weak solution}
	  The goal of this section is to show that the quadruple $(v,M,\phi,\mu)$ obtained in \eqref{weakconvergences} satisfies the energy inequality.
	  As  the first step we note that from convergences \eqref{VNStrongly}, \eqref{strngconvMnl8l4}, \eqref{MNStrongC}, \eqref{PhiNStrongly}, \eqref{weakconvergences}$_4$ we conclude that we have up to a nonrelabeled subsequences for a.a. $t\in(0,T)$
	 \begin{equation}\label{FixedTimeCnv}
	     \begin{alignedat}{2}
	         v^N(t)\to &v(t) &&\text{ in }L^2(\Omega),\\
	         M^N(t)\to &M(t) &&\text{ in }L^4(\Omega),\\
	         \nabla M^N(t)\to &\nabla M(t) &&\text{ in }L^2(\Omega),\\
	         \phi^N(t)\to &\phi(t) &&\text{ in }L^4(\Omega),\\
	         \nabla \phi^N(t)\rightharpoonup &\nabla \phi(t) &&\text{ in }L^2(\Omega).
	     \end{alignedat}
	 \end{equation}
	 Next we show that  
	 \begin{equation}\label{SemicontEnergy}
	    E_{tot}(v(t),M(t),\phi(t))\leq \liminf_{N\to \infty}E_{tot}(v^N(t),M^N(t),\phi^N(t))\text{ for a.a. }t\in(0,T).
	 \end{equation}
	 We first fix a $t\in(0,T)$ such that convergences from \eqref{FixedTimeCnv} are available. Next taking into account the definition of $E_{tot}$ in \eqref{defEtot}, we pass to the limit $N\to\infty$ in terms involving only $v^N(t)$ and $\phi^N(t)$ ($i.e$ the terms which corresponds to the first, fourth and fifth terms of \eqref{defEtot}), using corresponding convergences \eqref{FixedTimeCnv}$_{1,4}$ and the weak lower semicontinuity of the norm in combination with \eqref{FixedTimeCnv}$_5$. Consequently one has 
	 \begin{equation}\label{lowersemicontfewterms}
	 \begin{array}{ll}
	 & \displaystyle\frac{1}{2}\int_{\Omega}|v(t)|^{2}+\frac{\eta}{2}\int_{\Omega}|\nabla\phi(t)|^{2}+\frac{1}{4\eta}\int_{\Omega}(\phi(t)^{2}-1)^{2}\\
	 &\displaystyle\leqslant \liminf_{N\to \infty}\left(\frac{1}{2}\int_{\Omega}|v^{N}(t)|^{2}+\frac{\eta}{2}\int_{\Omega}|\nabla\phi^{N}(t)|^{2}+\frac{1}{4\eta}\int_{\Omega}(\phi^{N}(t)^{2}-1)^{2}\right).
	 \end{array}
	 \end{equation}
	 Next the strong convergence of $\{\phi^N(t)\}$, $\{\nabla M^N(t)\}$ and $\{M^N(t)\}$ respectively, imply the convergence a.e.\ in $\Omega$ for a nonrelabeled subsequence. This almost everywhere convergence along with \eqref{XiAssum}, the  strong convergences of $\{\nabla M^N(t)\}$, $\{M^N(t)\}$ respectively, and Lemma~\ref{Lem:GenDomConv} allow to pass to the limit $N\to\infty$ and show
	 \begin{equation}\label{convmix} 
	 \begin{split}
	 \int_{\Omega}\xi(\phi^{N})|\nabla M^{N}|^{2}&\longrightarrow\int_{\Omega}\xi(\phi)|\nabla M|^{2},\\
	 \frac{1}{4\alpha^{2}}\int_{\Omega}\xi(\phi^{N})(|M^{N}|^{2}-1)^{2}&\longrightarrow \frac{1}{4\alpha^{2}}\int_{\Omega}\xi(\phi)(|M|^{2}-1)^{2}.
	 \end{split}
	 \end{equation}
	Combining \eqref{lowersemicontfewterms} and \eqref{convmix} one shows \eqref{SemicontEnergy}.\\
	 Applying \eqref{MN0Cnv}, \eqref{PhiN0Cnv} and the embedding $W^{1,2}(\Omega)\hookrightarrow L^4(\Omega)$ we get 
	 \begin{equation}\label{TimeZeroConv}
	     E_{tot}(v_0,M^N_0,\phi^N_0)\to E_{tot}(v_0,M_0,\phi_0) 
	 \end{equation}
	 by similar arguments as above.	 Hence using \eqref{SemicontEnergy}, \eqref{TimeZeroConv}, \eqref{weakconvergences}$_{1,5}$, \eqref{WeakCnvMonster} and the weak lower semicontinuity of the norm we deduce from \eqref{energyinterpolant} that the following inequality holds for a.a. $t\in(0,T)$
	 \begin{equation}\label{EnIneq}
	 \begin{split}
	     E_{tot}(v(t),M(t),\phi(t))+&\int_0^t\left(\nu\|\nabla v\|^2_{L^2(\Omega)}+\|\nabla \mu\|^2_{L^2(\Omega)}+\left\|\dvr(\xi(\phi)\nabla M)-\frac{\xi(\phi)}{\alpha^2}M(|M|^2-1)\right\|^2_{L^2(\Omega)}\right)\\
	     &\leq E_{tot}(v_0,M_0,\phi_0).
	     \end{split}
	 \end{equation}
	 \subsection{Continuity in time of \texorpdfstring{$v,M,\phi$}{triple}}
	 This section is devoted to the improvement of the regularity of functions $v, M, \phi$ obtained in \eqref{weakconvergences}. Namely, our goal is to show that
	  \begin{equation}\label{TContinuityWeakTop}
	  \begin{split}
	      v&\in C_w([0,T];\LND),\\
	      M&\in C_w\left([0,T];W^{1,2}(\Omega)\right),\\
	      M&\in C\left([0,T];L^2(\Omega)\right),\\
	      \phi&\in C_w([0,T];W^{1,2}(\Omega)),\\
	      \phi&\in C([0,T];L^2(\Omega)).
	  \end{split}
	  \end{equation}
	  As a consequence of the uniform bound in \eqref{VNDiscTimeDerEst} we get $\tder v\in L^2\left(0,T;(V(\Omega))'\right)$ implying that $v\in C\left([0,T];(V(\Omega))'\right)$. Combining this fact with $v\in L^\infty(0,T;\LND)$ we conclude \eqref{TContinuityWeakTop}$_1$, cf.\ \cite[Ch. III, Lemma 1.4]{Tem77}. For the quantities $M$ and $\phi$ we deduce the regularity in \eqref {TContinuityWeakTop}$_{2,4}$ similarly, as $\tder M\in L^2\left(0,T;L^{\frac{3}{2}}(\Omega)\right)$ and $\tder\phi\in L^2\left(0,T;(W^{1,2}(\Omega))'\right)$  due to \eqref{TDerMNEst}, \eqref{TDerPhiNEst} respectively and \\
	  $M\in L^\infty\left(0,T;W^{1,2}(\Omega)\right)$, $\phi\in L^\infty(0,T;W^{1,2}(\Omega))$. In particular $M\in L^2(0,T;W^{1,2}(\Omega))$ along with $\tder M\in L^2(0,T;(W^{1,2}(\Omega))')$ implies \eqref {TContinuityWeakTop}$_3$ and \eqref {TContinuityWeakTop}$_5$ follows by the same argument. 
		\subsection{Recovering the weak formulations}\label{weakformobtain}
 This section is devoted to the justification of the fact that the quadruple $(v,M,\phi,\mu)$ obtained in \eqref{weakconvergences} satisfies the weak formulation in \eqref{WeakForm}. This will be achieved by passing to the limit $N\to\infty$ in integral identities \eqref{intidentity1}, \eqref{intidentity2}, \eqref{intidentity3} and \eqref{intidentity4}.\\
 Let use first consider \eqref{intidentity1} with the test function $\psi_{1}\in C^{1}_{c}([0,T);V(\Omega)).$ First of all one recalls from \eqref{AffInterpProp}$_1$ that
 $$\partial_{t}\widetilde{v}^{N}(t)=\partial^{-}_{t,h}v^{N}(t).$$
 From \eqref{TVNStrongly} and \eqref{tildVNWeaklyS} it follows that up to a not explicitly relabeled subsequence
 \begin{equation}\label{almostevrywhere}
    \widetilde{v}^{N}(t)\rightharpoonup v(t)\text{ in } L^{2}(\Omega)\text{ for a.a. }t\in(0,T).
  \end{equation}
   Fixing $\tau\in(0,T)$ such that \eqref{almostevrywhere} holds we integrate by parts in time to deduce from \eqref{intidentity1}:\\
 \begin{equation}\label{intidentity1*}
 \begin{split}
 & -\int_{0}^{\tau}\int_{\Omega} \widetilde{v}^{N}\cdot\partial_{t}\psi_{1}+\int_{\Omega}\widetilde{v}^{N}(\tau)\cdot\psi_{1}(\tau)-\int_{\Omega} \widetilde{v}^{N}(0)\cdot\psi_{1}(0)+\int_{0}^{\tau}\int_{\Omega}(v^{N}\cdot\nabla)v^{N}\cdot\psi_{1}\\
 &+\int_{0}^{\tau}\int_{\Omega}\nabla\mu^{N}\phi^{N}_{h}\cdot\psi_{1}
 +\int_{0}^{\tau}\int_{\Omega}\left(\mathrm{div}\left((\xi(\phi^{N}_{h})\nabla M^N\right)-\frac{\xi(\phi^{N}_{h})}{\alpha^{2}}(|M^{N}|^{2}M^{N}-M^{N}_{h})\right)\nabla M^{N}\cdot\psi_{1}
 \\
 &=-\nu\int_{0}^{\tau}\int_{\Omega}\nabla v^{N}\cdot\nabla\psi_{1}.
 \end{split}
 \end{equation}
  In view of \eqref{tildVNWeaklyS} one can pass to the limit $N\to\infty$ in the first integral of \eqref{intidentity1*}, whereas \eqref{almostevrywhere} allows for the limit passage in the second term of the left hand side of \eqref{intidentity1*}. 
  Further by definition of $\tilde v^N(0)$, cf.\ \eqref{AffInterpolantDef}, we have $\widetilde{v}^{N}(0)=v^{N}(-h)=v_{0}$ and for the third term on the left hand side of \eqref{intidentity1*} we get 
  \begin{equation*}
      \int_{\Omega} \widetilde{v}^{N}(0)\cdot\psi_1(0)=\int_{\Omega} v_0\cdot\psi_1(0).
  \end{equation*}
 Convergences \eqref{VNStrongly} and \eqref{weakconvergences}$_1$ suffice for the passage to the limit $N\to\infty$ in the fourth term of the left hand side of \eqref{intidentity1*}. The strong convergence \eqref{PhiNhStrongly} and the weak convergence \eqref{weakconvergences}$_5$ allows us to perform the limit passage in the fifth term on the left hand side of \eqref{intidentity1*}. One now recalls the weak convergence \eqref{WeakCnvMonster} and the strong convergence \eqref{MNStrongC} which suffices for the limit passage $N\to\infty$ in the fifth  integral of \eqref{intidentity1*}. Finally in view of the weak convergence \eqref{weakconvergences}$_1$, one passes to the limit in the final term of \eqref{intidentity1*}.\\[3.mm]
 Hence we reach the following expression:
 \begin{equation}\label{intidentity1**}
 \begin{split}
 & -\int_{0}^{\tau}\int_{\Omega} {v}\cdot\partial_{t}\psi_{1}+\int_{\Omega}{v}(\tau)\cdot\psi_{1}(\tau)-\int_{\Omega}  v_{0}\cdot\psi_{1}(0)+\int_{0}^{\tau}\int_{\Omega}(v\cdot\nabla)v\cdot\psi_{1}\\
 &+\int_{0}^{\tau}\int_{\Omega}\nabla\mu\phi\cdot\psi_{1}
 -\int_{0}^{\tau}\int_{\Omega}\left(\frac{\xi(\phi^{N}_{h})}{\alpha^{2}}(|M|^{2}-1)M\nabla M\right)\cdot\psi_{1}\\
 &+\int_{0}^{\tau}\int_{\Omega}\left(\mathrm{div}(\xi(\phi)\nabla M)\nabla M\right)\cdot\psi_{1}
 =-\nu\int_{0}^{\tau}\int_{\Omega}\nabla v\cdot\nabla\psi_{1}
 \end{split}
 \end{equation}
 for almost all $\tau\in (0,T)$. For fixed $t\in(0,T)$ we find a sequence $\{\tau_k\}$ such that $\tau_k\to t$ and \eqref{intidentity1**} holds with $\tau=\tau_k$. Then using \eqref{TContinuityWeakTop}$_1$ and the fact that all integrands in terms with the integration over time are integrable with respect to time we conclude that \eqref{intidentity1**} holds also for $t$. 
 This implies \eqref{WeakForm}$_{1}.$\\
 Now we perform the passage $N\to\infty$ in \eqref{intidentity2} in order to furnish the weak integral formulation \eqref{WeakForm}$_{2}.$ In that direction we first choose
 $\psi_{2}\in C^{1}_{c}([0,T);W^{1,2}(\Omega)))$ in \eqref{intidentity2}, further it follows from \eqref{AffInterpProp}$_1$ that
 $$\partial_{t}\widetilde{M}^{N}(t)=\partial^{-}_{t,h}M^{N}(t),$$
 and use integration by parts in time and space variables in \eqref{intidentity2} to infer
 \begin{equation}\label{intidentity2*}
 \begin{split}
 &-\int_{0}^{\tau}\int_{\Omega}\widetilde{M}^{N}\cdot\partial_{t}\psi_{2}+\int_{\Omega}\widetilde{M}^{N}(\tau)\cdot\psi_{2}(\tau)-\int_{\Omega}\widetilde{M}^{N}(0)\cdot\psi_{2}(0)+\int_{0}^{\tau}\int_{\Omega}(v^{N}\cdot\nabla)M^{N}\cdot\psi_{2}\\
 & =\int_{0}^{\tau}\int_{\Omega}-\xi(\phi^{N}_{h})\nabla
 M^{N}\cdot\nabla\psi_2-\frac{\xi(\phi^{N}_{h})}{\alpha^{2}}(|M^{N}|^{2}M^{N}-M^{N}_{h})\cdot\psi_2,
 \end{split}
 \end{equation}
 for a.a. $\tau\in(0,T)$. Next we fix $\tau$ and a not relabeled subsequence such that $\widetilde{M}^N(\tau)\to M(\tau)$ in $L^2(\Omega)$, which is possible due to \eqref{strngconvtM}. Hence we can pass to the limit $N\to\infty$ in the second term on the left hand side of \eqref{intidentity2*}. Moreover, \eqref{strngconvtM} allows us to pass to the limit in the first integral of \eqref{intidentity2*}. 
  As $\widetilde{M}^{N}(0)=M^N_0$, we get by \eqref{MN0Cnv} for the third term on the left hand side of \eqref{intidentity2*}
 \begin{equation*}
 \begin{split}
  \int_{\Omega}\widetilde{M}^{N}(0)\cdot\psi_{2}(0)\longrightarrow \int_{\Omega}M_{0}\cdot\psi_{2}(0).
 \end{split}
 \end{equation*}
 One can pass to the limit in the fourth integral of \eqref{intidentity2*} by using convergences \eqref{VNStrongly} and \eqref{MNStrongC}. We pass to the limit in the first term on the right hand side of \eqref{intidentity2*} by \eqref{NonlL2weakly}. For the passage to the limit in the last term of \eqref{intidentity2*} we apply \eqref{weakconvpolynom}. All the arguments presented above allow to obtain \eqref{WeakForm}$_2$ for a.a. $t\in (0,T)$. In view of \eqref{TContinuityWeakTop}$_{3}$ we deduce \eqref{WeakForm}$_2$ for all $t\in (0,T)$ and the proof follows the similar line of arguments used to show \eqref{WeakForm}$_{1}$ from \eqref{intidentity1**}.\\[3.mm]
 Now let us obtain \eqref{WeakForm}$_{3}$ from \eqref{intidentity3} with the test function $\psi_{3}\in C^{1}_{c}\left([0,T);W^{1,2}(\Omega)\cap L^\infty(\Omega)\right).$ In the similar spirit of \eqref{intidentity1*} and \eqref{intidentity2*} we write \eqref{intidentity3} as follows
 \begin{equation}\label{intidentity3*}
 \begin{split}
 &-\int_{0}^{\tau}\int_{\Omega} \widetilde{\phi}^{N}\partial_{t}\psi_{3}+\int_{\Omega}\widetilde{\phi}^{N}(\tau)\psi_{3}(\tau)-\int_{\Omega}\widetilde{\phi}^{N}(0)\psi_{3}+\int_{0}^{\tau}\int_{\Omega}(v^{N}\cdot\nabla)\phi^{N}_{h}\psi_{3}\\
 &=-\int_{0}^{\tau}\int_{\Omega}\nabla\mu^{N}\cdot\nabla\psi_{3},
 \end{split}
 \end{equation}
 for a.a. $\tau\in(0,T)$. Analogously to previous cases we pass to the limit $N\to\infty$ in the first and the second term on the left hand side of \eqref{intidentity3*} by employing \eqref{TildePhiNStrongly}. 
Furthermore, for the passage to the limit in the third term one uses the fact that $\widetilde \phi^N(0)=\phi^N_0$ and the convergence \eqref{PhiN0Cnv}. Next we note that due to \eqref{phiNh} and \eqref{PhiNhStrongly} we have up to a nonrelabeled subsequence
 \begin{equation*}
     \phi^N_h\rightharpoonup^*\phi\text{ in }L^\infty(0,T;W^{1,2}(\Omega)).
 \end{equation*}
 The latter convergence and the strong convergence in \eqref{VNStrongly} are sufficient for the limit passage in the fourth integral of \eqref{intidentity3*}, while the fifth and the final integrand is easily handled by the weak convergence \eqref{weakconvergences}$_5$. The arguments presented so far proves the identity \eqref{WeakForm}$_{3}$ for a.a. $t\in(0,T).$ In order to prove the identity for all $t\in(0,T)$ one proceeds similarly to the previous cases by employing \eqref{TContinuityWeakTop}$_5$.\\[3.mm]
 Now we perform the passage to the limit $N\to\infty$ in \eqref{intidentity4} to obtain \eqref{WeakForm}$_{4}$. We pass to the limit in the first terms on the left and right hand side of \eqref{intidentity4} by using convergences \eqref{weakconvergences}$_{4,5}$. We are going to explain more elaborately the limit passage in the rest of the terms of \eqref{intidentity4}. In order to deal with the second and third term of \eqref{intidentity4} we first recall the definition of $H_{0}$ from \eqref{H0} and infer:
 \begin{equation*}
 H_{0}(\phi^{N},\phi^{N}_{h})=\begin{cases}
 \xi{'}(\phi^{N}_h) &\mbox{if}\qquad \phi^{N}=\phi^{N}_h,\\
 \xi'(\zeta_{t,t-h}) & \mbox{if}\qquad \phi^{N}\neq \phi^{N}_h,
 \end{cases}
 \end{equation*}
 where $\zeta_{t,t-h}$ is an element of the line segment with endpoints $\phi^{N}$ and $\phi^{N}_h$. Since $\xi(\cdot)\in C^{1}(\mathbb{R})$ (cf.\ \eqref{XiAssum}$_{1}$) and $\phi^{N}$ and $\phi^{N}_{h}$ a.e.\ converge to $\phi$ by \eqref{PhiNStrongly},  \eqref{PhiNhStrongly} respectively, one has that
 \begin{equation}\label{aeconvergence}
 \begin{array}{l}
 H_{0}(\phi^{N},\phi^{N}_{h})\quad\mbox{converges a.e.\ to}\quad \xi'(\phi).
 \end{array}
 \end{equation}
 Once again one uses the upper bound \eqref{XiAssum}$_{3}$ of $\xi'$ and convergence \eqref{MNStrongC} to deduce by Lemma~\ref{Lem:GenDomConv} that
 $$H_{0}(\phi^{N},\phi^{N}_{h})|\nabla M^{N}|^{2}\longrightarrow \xi'(\phi)|\nabla M|^{2}\quad\mbox{in}\quad L^{1}(0,T;L^{1}(\Omega)),$$
 which is enough to pass limit in the second term on the left hand side of \eqref{intidentity4} since $\psi_{3}\in C^{1}_{c}([0,T);W^{1,2}(\Omega)\cap L^{\infty}(\Omega))$.
 
 Further convergence \eqref{strngconvMnl8l4} renders that
 $$(|M^{N}|^{2}-1)^{2}\longrightarrow (|M|^{2}-1)^{2}\quad\mbox{in}\quad L^{2}(0,T;L^{1}(\Omega)),$$
 which along with \eqref{aeconvergence} furnishes by Lemma~\ref{Lem:GenDomConv} that
 $$H_{0}(\phi^{N},\phi^{N}_{h})(|M^{N}|^{2}-1)^{2}\longrightarrow \xi'(\phi)(|M|^{2}-1)^{2}\quad\mbox{in}\quad L^{1}(0,T;L^{1}(\Omega)).$$
 Hence one concludes the passage to the limit in the third term on the left hand side of \eqref{intidentity4}.
 
 Let us now consider the last term of \eqref{intidentity4}. Convergence \eqref{PhiNStrongly} provides the strong convergence
 $$(\phi^{N})^{2}\longrightarrow\phi^{2}\quad\mbox{in}\quad L^{1}(0,T;L^{2}(\Omega)),$$
 which along with convergence \eqref{weakconvergences}$_4$ implies in particular the following weak convergence
 \begin{equation}\label{weakconvphi3}
 \begin{array}{l}
 (\phi^N)^{3}\rightharpoonup \phi^{3}\quad\mbox{in}\quad L^{1}(0,T;L^{1}(\Omega)).
 \end{array}
 \end{equation} 
Convergence \eqref{weakconvphi3} along with \eqref{PhiNhStrongly} suffices to pass to the limit in the second term on the right hand side of \eqref{intidentity4}. This finishes the obtainment of the integral identities in \eqref{WeakForm}.
	   \subsection{The attainment of initial data \texorpdfstring{$v_0, M_0, \phi_0$}{InitC}}\label{attainmentinitial}
	   In this section, we prove \eqref{InitDataAtt}. We begin with the proof of the following identities
	  \begin{equation}\label{InitValIdent}
	    \begin{alignedat}{2}
	    v(0)=&v_0&&\text{ a.e.\ in }\Omega,\\
	    M(0)=&M_0&&\text{ a.e.\ in }\Omega,\\
	    \phi(0)=&\phi_0&&\text{ a.e.\ in }\Omega.
	    \end{alignedat}
	  \end{equation}
	  The special choice of a test function $\psi_1=\omega_1\omega_2$, where $\omega_1\in C^1_c([0,T)),$ $\omega_{1}(0)>0$ and $\omega_2\in V(\Omega)$ are arbitrary but fixed, in  \eqref{WeakForm}$_1$ (which is already proved in Section \ref{weakformobtain}) yields
	  \begin{equation*}
	    \int_\Omega v_0\cdot\omega_1(0)\omega_2=\lim_{t\to 0_+}\int_\Omega v(t)\cdot\omega_1(t)\omega_2=\int_\Omega v(0)\cdot\omega_1(0)\omega_2,
	  \end{equation*}
	  where the second equality follows by \eqref{TContinuityWeakTop}$_1.$The latter identity with the special choice of $\omega_2=v_0-v(0)$ which is possible due to the density of $V(\Omega)$ in $L^2_{\dvr}(\Omega)$ implies \eqref{InitValIdent}$_1$. The remaining identities from \eqref{InitValIdent} are shown by repeating the above arguments.
	  
	  Taking into account the regularity from \eqref{TContinuityWeakTop} we now claim that the energy inequality \eqref{EnIneq} holds for all $t\in[0,T]$. To this end we consider an arbitrary $t\in[0,T]$ and a sequence $\{t^k\}$ such that $t^k\geq t$, $t^k\to t$ as $k\rightarrow\infty$ and 
	   \begin{equation}\label{tkconv}
	   \begin{alignedat}{2}
	  \phi(t^{k})&\rightharpoonup \phi(t)&&\mbox{ in }W^{1,2}(\Omega),\\
	  \phi(t^{k})&\rightarrow \phi(t)&&\mbox{ in }L^{2}(\Omega),\\
	  M(t^{k})&\rightharpoonup M(t)&&\mbox{ in } W^{1,2}(\Omega),\\
	  M(t^{k})&\rightarrow M(t)&&\mbox{ in }L^{2}(\Omega).\\
	  \end{alignedat}
	  \end{equation}
	  and \eqref{EnIneq} holds for each $t^k$.
	  Consequently
	  \begin{equation}\label{phitkcnv}
	    \xi(\phi(t^k,\cdot))\to\xi(\phi(t,\cdot))\text{ a.e.\ in }\Omega\text.
	   \end{equation}
	    We note that the existence of such a sequence $\{t^k\}$ follows by \eqref{TContinuityWeakTop}$_{4,5}$. Due to the weak lower semicontinutiy of the norm in $L^2(\Omega),$ $L^{4}(\Omega)$ and \eqref{tkconv}$_{1,2}$, we obtain
	  \begin{equation}\label{1stPartEnLSC}
	  \begin{split}
	      &\liminf_{k\to\infty}\int_\Omega\left( \frac{\eta}{2}|\nabla\phi(t^k)|^2+\frac{1}{4\eta} (\phi(t^k)^2-1)^2\right)=\liminf_{k\to\infty}\int_\Omega\left(\frac{\eta}{2}|\nabla\phi(t^k)|^2+\frac{1}{4\eta} \left((\phi(t^k))^4-2(\phi(t^k))^2+1\right)\right)\\
	      &\geq \int_\Omega\left(\frac{\eta}{2}|\nabla\phi(t)|^2+\frac{1}{4\eta} \left( (\phi(t))^4-2(\phi(t))^2+1\right)\right)=\int_\Omega \left(\frac{\eta}{2}|\nabla\phi(t)|^2+\frac{1}{4\eta}(\phi(t)^2-1)^2\right).
	      \end{split}
	  \end{equation}
	  Next, thanks to the convexity of the function $|\cdot|^p$, $p\geq 1$ we have $|A|^p-|B|^p\geq p|B|^{p-2}B\cdot(A-B)$ for $A,B\in\mathbb{R}^m$. Using this fact, \eqref{tkconv}$_{3,4},$ \eqref{phitkcnv}, the Lebesgue dominated convergence theorem and \eqref{XiAssum}$_{2},$ we obtain
	  \begin{equation}\label{2ndPartEnLSC}
	  \begin{split}
	      &\liminf_{k\to\infty} \int_{\Omega}\left(\xi(\phi(t^k))|\nabla M(t^k)|^2+\frac{\xi(\phi(t^k))}{\alpha^2}(|M(t^k)|^2-1)^2\right)\\
	      &\geq \liminf_{k\to\infty} \int_{\Omega}\Bigl(\xi(\phi(t^k))\left(|\nabla M(t)|^2+2\nabla M(t)\cdot\left(\nabla M(t^k)-\nabla M(t)\right)\right)\\&+\frac{\xi(\phi(t^k))}{\alpha^2}\left(|M(t)|^4+4|M(t)|^2M(t)\cdot\left(M(t^k)-M(t)\right)-2|M(t^k)|^2+1\right)\Bigr)\\
	      &=\int_{\Omega}\left(\xi(\phi(t))|\nabla M(t)|^2+\frac{\xi(\phi(t))}{\alpha^2}(|M(t)|^2-1)^2\right)
	      \end{split}
	  \end{equation}
	   Altogether \eqref{1stPartEnLSC}, \eqref{2ndPartEnLSC} and \eqref{TContinuityWeakTop}$_1$ imply that \eqref{EnIneq} holds for all $t\in[0,T]$ and hence our claim.
	  
	  Hence we obtain
	  \begin{equation}\label{LSIneq}
	      \limsup_{t\to 0_+}E_{tot}(v(t),M(t),\phi(t))\leq E_{tot}(v_0,M_0,\phi_0).
	  \end{equation}
	  On the other hand using \eqref{TContinuityWeakTop} along with \eqref{InitValIdent} we get like at \eqref{1stPartEnLSC} and \eqref{2ndPartEnLSC}
	  \begin{equation*}
	      \liminf_{t\to 0_+}E_{tot}(v(t),M(t),\phi(t))\geq E_{tot}(v(0),M(0),\phi(0))=E_{tot}(v_0,M_0,\phi_0),
	  \end{equation*}
	  which in a combination with \eqref{LSIneq} yields
	  \begin{equation}\label{LimEn0}
	      \lim_{t\to 0_+}E_{tot}(v(t),M(t),\phi(t))= E_{tot}(v_0,M_0,\phi_0).
	  \end{equation}
	  Going back to the definition of $E_{tot}$ and employing the strong convexity of $|\cdot|^2$, i.e., $|A|^2-|B|^2\geq 2B\cdot(A-B)+2|A-B|^2$ for all $A,B\in\mathbb{R}^m$ and the convexity of $|\cdot|^4$ ($i.e.$ the inequality $|A|^4-|B|^4\geq 4|B|^{2}B\cdot(A-B)$ for all $A,B\in\mathbb{R}^m$),
	  it follows that for each $t\in(0,T)$
	  \begin{equation}\label{ineqdetot}
	  \begin{split}
	      E&_{tot}(v(t),M(t),\phi(t))-E_{tot}(v_0,M_0,\phi_0)\\
	      \geq&  \frac{1}{2}\int_\Omega 2v_0\cdot(v(t)-v_0)+\int_\Omega|v(t)-v_0|^2+\int_\Omega\left(\xi(\phi(t))-\xi(\phi_0)\right)|\nabla M_0|^2\\
	       &+\int_\Omega 2\xi(\phi(t))\nabla M_0\cdot(\nabla M(t)-\nabla M_0)+2\int_\Omega \xi(\phi(t))|\nabla M(t)-\nabla M_0|^2\\
	      &+\frac{1}{4\alpha^2}\int_\Omega \left(\xi(\phi(t))-\xi(\phi_0)\right)|M_0|^4 +\frac{1}{\alpha^2}\int_\Omega \xi(\phi(t))|M_0|^2M_0\cdot\left(M(t)-M_0\right)\\
	      &-\frac{1}{2\alpha^2}\int_\Omega  \left(\xi(\phi(t))|M(t)|^2-\xi(\phi_0)|M_0|^2\right)+\frac{1}{4\alpha^2}\int_\Omega  \left(\xi(\phi(t))-\xi(\phi_0)\right)\\
	      &+\eta\int_\Omega\nabla\phi_0\cdot\left(\nabla\phi(t)-\nabla\phi_0\right)+\eta\int_\Omega|\nabla\phi(t)-\nabla\phi_0|^2+\frac{1}{\eta}\int_\Omega\phi_0^3\left(\phi(t)-\phi_0\right) \\&-\frac{1}{2\eta}\int_\Omega \left(\phi(t)^2-\phi_0^2\right)=\sum_{m=1}^{13}I_m(t).
	  \end{split}
	  \end{equation}
	  The task now is to prove \eqref{InitDataAtt} by taking the limsup $t\to 0_+$ on both sides of the inequality \eqref{ineqdetot}. To this end we consider an arbitrary sequence $\{t^k\}$ such that $t^k\to 0_+$ as $k\to\infty$. The sequence $\{t^k\}$ possesses a subsequence $\{t^{k'}\}$ such that 
	  \begin{equation}\label{PhiTkPPcnv}
	    \phi(t^{k'})\to \phi_0,\ M(t^{k'})\to M_0\text{ a.e.\ in }\Omega\text{ as }k'\to\infty
	  \end{equation}
	  by \eqref{TContinuityWeakTop}$_{3,5}$ and \eqref{InitValIdent}$_{2,3}$.
	  Then using \eqref{PhiTkPPcnv}, \eqref{XiAssum} and the Lebesgue dominated convergence theorem we conclude that
	  \begin{equation}\label{LimImI}
	    \lim_{k^{'}\to\infty}I_m(t^{k'})=0\text{ for }m=3,6,9.
	  \end{equation}
	  Employing additionally \eqref{TContinuityWeakTop}$_{2,3}$, \eqref{XiAssum}$_{2}$ and \eqref{InitValIdent} we deduce
	  \begin{equation}\label{LimImII}
	    \lim_{k^{'}\to\infty}I_m(t^{k'})=0\text{ for }m=4,7.
	  \end{equation}
	  	By \eqref{TContinuityWeakTop}$_3$, \eqref{PhiTkPPcnv} and Lemma~\ref{Lem:GenDomConv} we infer 
	  	  \begin{equation}\label{LimImIII}
	    \lim_{k^{'}\to\infty}I_8(t^{k'})=0.
	  \end{equation}
	  Moreover, it immediately follows from \eqref{TContinuityWeakTop}$_{1,4,5}$, \eqref{InitValIdent} that
	  \begin{equation}\label{LimImIV}
	    \lim_{k^{'}\to\infty}I_m(t^{k'})=0\text{ for }m=1,10,12,13.
	  \end{equation}
	  Using \eqref{LimImI}--\eqref{LimImIV} we deduce from \eqref{LimEn0} and \eqref{ineqdetot} that
	  \begin{equation*}
	      \limsup_{k '\to\infty}\left(I_2(t^{k'})+I_5(t^{k'})+I_{11}(t^{k'})\right)\leq 0.
	  \end{equation*}
	  Applying \eqref{XiAssum}$_2$ in the latter identity we arrive at
 	  \begin{equation*}
	      \limsup_{k '\to\infty} \left(\frac{1}{2}\|v(t^{k'})-v_0\|^2_{L^2(\Omega)}+c_1\|\nabla M(t^{k'})-\nabla M_0\|^2_{L^2(\Omega)}+\eta\|\nabla\phi(t^{k'})-\nabla\phi_0\|^2_{L^2(\Omega)}\right)\leq 0.
	  \end{equation*}
	  Combining the above inequality with \eqref{TContinuityWeakTop}$_{3,5}$ we get 
	  \begin{equation}\label{LimTkP}
	      \lim_{k '\to\infty}\left(\|v(t^{k'})-v_0\|_{L^2(\Omega)}+\|M(t^{k'})- M_0\|_{W^{1,2}(\Omega)}+\|\phi(t^{k'})-\phi_0\|_{W^{1,2}(\Omega)}\right)=0.
	  \end{equation} 
	 As $\{t^{k}\}$ was selected arbitrarily and possesses a subsequence that satisfies the latter identity, \eqref{InitDataAtt} has to hold. Indeed, if \eqref{InitDataAtt} were not true, one could find a sequence $\{t^{\bar{k}}\}$ with $t^{\bar{k}}\to 0_+$ from which can not be selected a subsequence satisfying \eqref{LimTkP}. This contradicts our finding that any sequence $\{t^k\}$ with $t^k\to 0$ possesses a subsequence $\{t^{k'}\}$ for which \eqref{LimTkP} holds.
\subsection{Conclusion}	Here we gather the results obtained so far to conclude the proof of Theorem \ref{Thm:Main}.
\begin{itemize}
    \item The inclusion in functional spaces (cf.\ \eqref{DefWS}) which are weakly continuous and continuous in time follows from \eqref{TContinuityWeakTop}. That $v$ belongs to $L^{2}(0,T;W^{1,2}_{0,\mbox{div}}(\Omega))$ and $\mu$ belongs to $L^{2}(0,T;W^{1,2}(\Omega))$ follow respectively from \eqref{weakconvergences}$_{1}$ and \eqref{weakconvergences}$_{5}.$ In order to complete the proof of \eqref{DefWS} we just have to show that $M\in W^{1,2}(0,T;L^{\frac{3}{2}}(\Omega)).$ For that in view of \eqref{TildeMNBound} one first observes that up to a nonrelabeled subsequence, $\widetilde{M}^{N}$ weakly converges in the space $W^{1,2}(0,T;L^{\frac{3}{2}}(\Omega))$ and using \eqref{strngconvtM} this weak limit can be identified with $M.$ Consequently $M\in W^{1,2}(0,T;L^{\frac{3}{2}}(\Omega)).$ This concludes the proof of \eqref{DefWS}.
    \item The weak formulation \eqref{WeakForm} solved by $(v,M,\phi,\mu)$ is proved in Section \ref{weakformobtain}.
    \item The attainment of the initial data in the sense of \eqref{InitDataAtt} is obtained in Section \ref{attainmentinitial}.
\end{itemize}
In view of the above items we finally conclude the proof of Theorem~\ref{Thm:Main}.  
	 
	   \section{Further comments}\label{sec:comments}
	 In this section we would like to comment on how to adapt our strategy in order to incorporate non degenerate variable viscosity and mobility in system~\eqref{diffviscoelastic*}. We will also discuss an extension of the model \eqref{diffviscoelastic*} in case the fluids under consideration are viscoelastic in nature and the elastic behavior is modeled by a regularized equation for the deformation gradient.  
	 \subsection{The case of non degenerate variable viscosity and mobility} It is natural to consider that the viscosity and mobility coefficients are functions of the order parameter, i.e., $\nu=\nu(\phi)$ and $m=m(\phi)$ and in this case the terms $\nu\Delta v$ in \eqref{diffviscoelastic*}$_{1}$ and $\Delta\mu$ in \eqref{diffviscoelastic*}$_{4}$ are replaced respectively by $\mbox{div}\,(2\nu(\phi)\mathbb{D}(v))$ where $\mathbb{D}(v)=\frac{1}{2}\left(\nabla v+(\nabla v)^\top\right)$ and $\dvr(m(\phi)\nabla\mu).$ We further assume the smoothness, non degeneracy and boundedness of these coefficients: 
	 \begin{equation}\label{condbndul}
	 \begin{array}{l}
	 m\in C^{1}(\mathbb{R}),\,\,\nu\in C^{0}({\mathbb{R}}),\,\,0<K_{0}<\nu(\cdot),\,\,m(\cdot)<K_{1},
	 \end{array}
	 \end{equation}
	 for some positive constants $K_{0}$ and $K_{1}.$ The energy dissipation in this case takes the form:
	 \begin{equation}\label{energyvariable}
	 \begin{split}
	 \frac{d}{dt}E_{tot}(v,M,\phi)=& -\left(\int_{\Omega}2\nu(\phi)|\mathbb{D} (v)|^{2}
	 +\int_{\Omega}m(\phi)|\nabla\mu|^{2}\right.\\ &\left.+\int_{\Omega}\left|\mathrm{div}(\xi(\phi)\nabla M)-\frac{\xi(\phi)}{\alpha^{2}}(|M|^{2}-1)M\right|^{2}\right),
	 \end{split}
	 \end{equation}
	 where $E_{tot}(v,M,\phi)$ is given by \eqref{defEtot}. One can note that a bound of $v$ in $L^{2}(0,T;W^{1,2}_{0,\dvr}(\Omega)),$ for any $T>0,$ can be obtained from the formal energy identity \eqref{energyvariable} by using Korn's inequality.\\
	 One can prove a result on global existence of weak solutions for this model exactly in the same functional framework \eqref{functionalspaces}. One can proceed by adapting the time discretization used in Section \ref{sec2}. Under the assumptions \eqref{XiAssum}, \eqref{assumk} and \eqref{condbndul} there is no particular difficulty to prove a existence result of the form Theorem~\ref{existenceweaksoldiscrete} and recover the solution in the space \eqref{regularityk1} for the time discretized problem. In \eqref{Lk} one just needs to define the duality pairing $\langle \mathcal{A}v,\widetilde{\psi}_{1} \rangle$ as
	 $$\langle\mathcal{A}v,\widetilde\psi_1\rangle=\int_{\Omega}2\nu(\phi_{k})\mathbb{D}(v)\cdot\mathbb{D} \widetilde\psi_1\text{ for all }\widetilde\psi_1\in\WND$$
	 and replace $\Delta_{N}\phi$ by $\mbox{div}_{N}(m(\phi_{k})\nabla\phi)$ where:
$$
	 	\langle-\mbox{div}_{N}(m(\phi_{k})\nabla\phi,\widetilde\psi_3\rangle=\int_{\Omega}m(\phi_{k})\nabla\phi\cdot\nabla\widetilde\psi_3\text{ for all }\widetilde\psi_3\in W^{1,2}(\Omega),
	 $$
	and perform minor modifications in what follows.\\
	The line of arguments used in Section \ref{sec3} can also be adapted with minor modifications to pass from the discretized setting to the original time dependent problem. In particular, we point out that estimates following from the energy inequality for interpolants along with arguments similar to that ones leading to \eqref{NonlL2weakly} can be used to infer the following convergences:
	\begin{equation*}
	\nu(\phi_h^N)\mathbb{D} (v^N)\rightharpoonup \nu(\phi)\mathbb{D} (v)\text{ in }L^2(Q_T)\,\,\mbox{and}\,\,m(\phi_h^N)\nabla \mu^N\rightharpoonup m(\phi)\nabla \mu\text{ in }L^2(Q_T),
	\end{equation*} 
	which are sufficient for the limit passage $N\to\infty$ in the new terms appearing in the weak formulation as contributions from $\mbox{div}\,(2\nu(\phi)\mathbb{D}(v))$ and $\mbox{div}(m(\phi)\nabla\mu).$
	\subsection{The case of viscoelastic fluids} Here we assume the diffuse interface bi fluid system is viscoelastic in nature and following the article \cite{liubenesova}, the elastic behavior is modeled by using a regularized evolution equation for the deformation gradient $F:Q_{T}\longrightarrow\mathbb{R}^{d\times d},$\\
	\begin{equation}\label{dgradient}
	\begin{split}
	 \partial_{t}F+(v\cdot\nabla)F-\nabla vF=\kappa\Delta F&\quad\mbox{in}\quad Q_{T},\\
	 F=0&\quad\mbox{on}\quad \Sigma_{T},\\
	F(\cdot,0)=F_{0}(\cdot)=\mathbb{I}&\quad\mbox{in}\quad \Omega,
	\end{split}
	\end{equation}
	where $\mathbb{I}$ is the $d\times d$ identity matrix and the regularizing coefficient $\kappa$ is a positive constant and can be arbitrarily small. The momentum equation \eqref{diffviscoelastic*}$_{1}$ should now be modified due to the contribution in the stress tensor from the elastic energy and now it reads as follows:
	\begin{equation}\label{modifiedmomentum}
	\begin{split}
	&\partial_{t}v+(v\cdot\nabla)v-\nu\Delta v+\nabla p\\
	&=\mu\nabla\phi+\frac{\xi(\phi)}{\alpha^{2}}((|M|^{2}-1)M)\nabla M-\dvr(\xi(\phi)\nabla M)\nabla M+\mbox{div}(FF^{T})\,\,\mbox{in}\,\,Q_{T}.
	\end{split}
	\end{equation} 
	The modified system \eqref{modifiedmomentum}-\eqref{diffviscoelastic*}$_{2}$-\eqref{diffviscoelastic*}$_{3}$-\eqref{diffviscoelastic*}$_{4}$-\eqref{diffviscoelastic*}$_{5}$-\eqref{diffviscoelastic*}$_{6}$-\eqref{diffviscoelastic*}$_{7}$-\eqref{diffviscoelastic*}$_{8}$-\eqref{dgradient} admits of a energy dissipation of the form
	\begin{equation}\nonumber
	\begin{split}
	\frac{d}{dt}\overline{E}_{tot}(v,M,F,\phi)=& -\left(\int_{\Omega}|\nabla v|^{2}
	+\int_{\Omega}|\nabla\mu|^{2}+\kappa\int_{\Omega}|\nabla F|^{2}\right.\\ &\left.+\int_{\Omega}\left|\mathrm{div}(\xi(\phi)\nabla M)-\frac{\xi(\phi)}{\alpha^{2}}(|M|^{2}-1)M\right|^{2}\right),
	\end{split}
	\end{equation}
	where 
		\begin{equation}\nonumber
		\begin{split} \overline{E}_{tot}(v,M,F,\phi)=&\frac{1}{2}\int_{\Omega}|v|^{2}+\int_{\Omega}\xi(\phi)|\nabla M|^{2}+\frac{1}{4\alpha^{2}}\int_{\Omega}\xi(\phi)(|M|^{2}-1)^{2}\\
		&+\frac{1}{2}\int_{\Omega}|F|^{2}+\frac{\eta}{2}\int_{\Omega}|\nabla\phi|^{2}+\frac{1}{4\eta}\int_{\Omega}(\phi^{2}-1)^{2}.
		\end{split}
		\end{equation}
		One can solve system \eqref{modifiedmomentum}-\eqref{diffviscoelastic*}$_{2}$-\eqref{diffviscoelastic*}$_{3}$-\eqref{diffviscoelastic*}$_{4}$-\eqref{diffviscoelastic*}$_{5}$-\eqref{diffviscoelastic*}$_{6}$-\eqref{diffviscoelastic*}$_{7}$-\eqref{diffviscoelastic*}$_{8}$-\eqref{dgradient} by adapting our time discretization strategy with suitable modifications. In view of the regularization of the evolution equation for the deformation gradient one has enough compactness for the limit passage in new non linear terms appearing in the weak formulation as a contribution from $\mbox{div}(FF^{T}),$ $(v\cdot\nabla)F$ and $\nabla vF.$ Consequently the system can be solved in the functional framework \eqref{functionalspaces} along with:
		$$F\in C_w([0,T];L^2(\Omega))\cap L^2(0,T;W^{1,2}(\Omega)).$$
\appendix
	\section{}
	In this section we collect several auxiliary assertions. The following lemma deals with the improvement of the regularity of solution to the system of Poisson equations.
	\begin{lemma}\label{Lem:PoissReg}
	Let $\Omega\subset\eR^d$ be a bounded domain of class $C^{1,1}$, $f\in L^p(\Omega)^m$,\ $m\in\eN$, $p\in(1,\infty)$ and $u:\Omega\to\eR^m$ be a solution to $\Delta u=f$ in $\Omega$, $\partial_n u=0$ on $\partial\Omega$. Then there is $c=c(p,\Omega,m)>0$
	\begin{equation}\label{W2Est}
	    \|u\|_{W^{2,p}(\Omega)}\leq c(\|u\|_{W^{1,p}(\Omega)}+\|f\|_{L^p(\Omega)}).
	\end{equation}
	\begin{proof}
	   Let us fix $i\in \{1,\ldots,m\}$. Then $u_i$ fulfills the equation $\Delta u_i=f_i$ and we have 
	   \begin{equation*}
	       \|u_i\|_{W^{2,p}(\Omega)}\leq c(\|f_i\|_{L^p(\Omega)}+\|u_i\|_{W^{1-\frac{1}{p},p}(\partial(\Omega))}+\|u_i\|_{W^{1,p}(\Omega)}),
	   \end{equation*}
	   cf.\ \cite[Section 2.3.3]{Gr85}. The continuity of the trace operator from $W^{1,p}(\Omega)$ to $W^{1-\frac{1}{p},p}(\partial\Omega)$ yields
	   	 \begin{equation*}
	       \|u_i\|_{W^{2,p}(\Omega)}\leq c(p,\Omega)(\|f_i\|_{L^p(\Omega)}+\|u_i\|_{W^{1,p}(\Omega)}).
	   \end{equation*}
	   Finally, we sum over $i\in\{1,\ldots m\}$ in the latter inequality to conclude \eqref{W2Est}.
	\end{proof}
	\end{lemma}
The following lemma slightly generalizes the well known dominated convergence theorem to the case of a sequence with convergent majorants.
\begin{lemma}\label{Lem:GenDomConv}
Let $\{f^n\}$ be a sequence of measurable functions on a measurable set $E$ converging a.e.\ in $E$ to $f$ and $\{g^n\}$ be a sequence of nonnegative measurable functions converging a.e.\ in $E$ to $g$ such that 
\begin{equation}\label{GnMajFn}
|f^n|\leq g^n\text{ for each }n\in\eN
\end{equation} 
and $\lim_{n\to\infty}\int_Eg^n=\int_E g<\infty$ then $f^n\to f$ in $L^1(E)$.
\begin{proof}
	The pointwise convergences of $\{f^n\}$, $\{g^n\}$ imply the measurability of limits $f,g$. Moreover, due to \eqref{GnMajFn} $|f|\leq g$ a.e.\ in $E$ follows. Hence $\{g^n+g-|f^n-f|\}$ is a sequence of nonnegative measurable functions such that $g^n+g-|f^n-f|\to 2g$ a.e.\ in $E$. By the Fatou lemma we have
	\begin{equation*}
		\int_E2g=\int_E \liminf_{n\to\infty}(g^n+g-|f^n-f|)\leq \liminf_{n\to\infty}\int_E\left( g^n+g-|f^n-f|\right)=\int_E 2g-\limsup_{n\to\infty}\int_E|f^n-f|.
	\end{equation*}
	We note that the last equality follows as $\int_E g^n\to\int_E g$ by assumption.
	Therefore $0\leq\liminf_{n\to\infty}\int_E|f^n-f|\leq \limsup_{n\to\infty}\int_E|f^n-f|\leq 0$ and we conclude $f^n\to f$ in $L^1(E)$.
\end{proof}
\end{lemma}
The following variant of Poincar\'e's inequality follows directly from \cite[Ch. 1, Theorem 1.5]{Ne2012}.
\begin{lemma}\label{Lem:PoincareAverage}
Let $\Omega\subset\eR^d$ be a bounded domain with continuous boundary. Then there is $c>0$ such that
\begin{equation*}
    \|u\|_{W^{1,2}(\Omega)}\leq c\left(\|\nabla u\|_{L^2(\Omega)}+\left|\int_\Omega u\right|\right)\text{ for all }u\in W^{1,2}(\Omega).
\end{equation*}
\end{lemma}
We will use several times the compactness result in the ensuing lemma, for its proof see \cite[Theorem 5]{Sim87}.
\begin{lemma}\label{Lem:RelComp}
    Let $T>0$, $p\in[1,\infty]$ and Banach spaces $X_1,X_2,X_3$ satisfy $X_1\stackrel{C}{\hookrightarrow}X_2\hookrightarrow X_3$. Assume that $F\subset L^p(0,T;X_1)$ fulfills
    \begin{enumerate}
        \item $\sup_{f\in F}\|f\|_{L^p(0,T X_1)}<\infty$,
        \item $\sup_{f\in F}\|\tau_s f-f\|_{L^p(0,T-s;X_3)}\to 0$ as $s\to 0$.
    \end{enumerate}
    Then $F$ is relatively compact in $L^p(0,T;X_2)$ and $C([0,T];X_2)$ if $p=\infty$.
\end{lemma}
The following lemma is a particular case of a more abstract result, cf.\ \cite[Lemma 9.1]{Alt12}.
\begin{lemma}\label{Lem:RefIneq}
    Let $N\in\eN$, a Hilbert space $H$ and $\{u_k\}\subset H$ be given. Moreover, assume that the function $u^N$ being defined via $u^N(t)=u_k$ for $t\in [(k-1)h,kh)$, $k\in\eN$ with $h=\frac{1}{N}$, satisfies
    \begin{equation}\label{BoundIntDiff}
        \int_0^{mh-s}\|u^N(t+s)-u^N(t)\|^2_H\dt\leq c s^{q}
    \end{equation}
    where $s=lh$, $l\in\eN$, $l\leq m$ and $q\in(0,1]$. Then \eqref{BoundIntDiff} holds with any $s>0$.
\end{lemma}

\noindent{\textbf{Acknowledgment}.} This work is funded by the Deutsche Forschungsgemeinschaft (DFG, German Research Foundation), grant SCHL 1706/4-2, project number 391682204.

\bibliographystyle{plain}

	\end{document}